\documentclass[10pt]{amsart}
\usepackage{amssymb}
\usepackage{verbatim}
\usepackage[usenames]{color}
\usepackage{hyperref}
\usepackage{url}
\usepackage[all]{xy}

\newtheorem{theorem}{Theorem}[section]
\newtheorem{proposition}[theorem]{Proposition} 
\newtheorem{lemma}[theorem]{Lemma}
\newtheorem{cor}[theorem]{Corollary}
\newtheorem{fact}[theorem]{Fact}
\newtheorem{observation}[theorem]{Observation}
\newtheorem{counterexample}[theorem]{Counterexample}

\theoremstyle{definition}
\newtheorem{definition}[theorem]{Definition}
\newtheorem{convention}[theorem]{Convention}

\newtheorem{question}[theorem]{Question}
\newtheorem{assumption}[theorem]{Assumption}

\theoremstyle{remark}

\newcommand{\defemph}{\textit}

\newcommand{\dom}{\mathrm{Dom}}

\newcommand{\cf}{\mathrm{cf}}
\newcommand{\mf}{\mathfrak}
\newcommand{\mc}{\mathcal}
\newcommand{\mbb}{\mathbb}

\newcommand{\forces}{\Vdash}

\newcommand{\om}{\omega}
\newcommand{\Fin}{\mathit{fin}}
\newcommand{\lgl}{\langle}
\newcommand{\rgl}{\rangle}
\newcommand{\al}{\alpha}
\newcommand{\ra}{\rightarrow}
\newcommand{\bP}{\mathbb{P}}
\newcommand{\bB}{\mathbb{B}}

\newcommand{\Vopenka}{Vop\v{e}nka}
\newcommand{\Bukovsky}{Bukovsk\'{y}}
\newcommand{\Coplakova}{Copl\'{a}kov\'{a}}

\title{Perfect Tree Forcings for Singular Cardinals}
\author{Natasha Dobrinen, Dan Hathaway,
 and Karel Prikry}

\address{Department of Mathematics, 
University of Denver, 2280 S Vine St, Denver, CO 80208, USA}
\email{Natasha.Dobrinen@du.edu}

\address{
Department of Mathematics \\
University of Vermont\\
Innovation Hall \\
82 University Place\\
Burlington, VT 05405 U.S.A.}
\email{Daniel.Hathaway@uvm.edu}

\address{Department of Mathematics, University of Minnesota, 127 Vincent Hall, 206 Church St.\ SE, Minneapolis, MN 55455}
\email{prikry@math.umn.edu}

\thanks{Dobrinen's research was partially supported by National Science Foundation Grant DMS-1600781.}

\begin{document}

\begin{abstract}
We investigate forcing properties of 
perfect tree forcings defined by Prikry to answer a question of Solovay in the late 1960's regarding first failures of distributivity.
Given a strictly increasing sequence of regular cardinals $\lgl \kappa_n:n<\om\rgl$, 
Prikry defined 
the forcing $\bP$  of all
 perfect subtrees of $\prod_{n<\om}\kappa_n$,
and proved that
  for $\kappa=\sup_{n<\om}\kappa_n$,
assuming the necessary cardinal arithmetic,
the Boolean completion  $\bB$ of $\bP$ is $(\om,\mu)$-distributive  for all $\mu<\kappa$
 but  $(\om,\kappa,\delta)$-distributivity fails for all $\delta<\kappa$,
implying failure of  the $(\om,\kappa)$-d.l.
 These hitherto unpublished results  are included, setting the stage for the following recent results.
$\bP$ satisfies a Sacks-type property, implying that $\bB$ is  $(\om,\infty,<\kappa)$-distributive.
The
 $(\mathfrak{h},2)$-d.l.\ and the $(\mathfrak{d},\infty,<\kappa)$-d.l.\ fail  in $\bB$.
 $\mathcal{P}(\om)/\Fin$ completely embeds into $\bB$.
 Also, $\bB$ collapses $\kappa^\omega$ to $\mathfrak{h}$.
We further prove that if $\kappa$ is a limit of countably many measurable cardinals,
then $\bB$ adds a minimal degree of constructibility for new $\om$-sequences.
Some of  these results generalize to cardinals $\kappa$ with uncountable cofinality.
\end{abstract}

\maketitle

\section{Introduction}\label{intro}

An ongoing area of research is to
 find complete Boolean algebras that witness first failures of distributive laws.
In the late 1960's, Solovay asked the following question: 
For which cardinals $\kappa$ 
is there a complete Boolean algebra $\mathbb{B}$ such that for all $\mu<\kappa$, the $(\om,\mu)$-distributive law holds in $\mathbb{B}$, while the $(\om,\kappa)$-distributive law fails (see \cite{Namba70})?
In forcing language, Solovay's question asks for which 
cardinals $\kappa$ is there a forcing extension in which 
there is a new $\om$-sequence of ordinals in $\kappa$, while every $\om$-sequence of ordinals bounded below $\kappa$ is in the ground model?
Whenever such a Boolean algebra exists, it must be the case that
$\mu^{\om}<\kappa$, for all $\mu<\kappa$.
It also must be the case that
 either $\kappa$ is regular or else $\kappa$ has cofinality $\om$, as shown in \cite{Namba70}.

For the case when $\kappa$ is regular, Solovay's question was solved independently using different forcings by Namba in \cite{Namba70} and \Bukovsky\ in \cite{Bukovsky75}.
Namba's forcing is similar to Laver forcing, where above the stem, all nodes split with the number of immediate successors having maximum cardinality.
  \Bukovsky's   forcing consists of perfect  trees, where splitting nodes have the maximum cardinality of  immediate successors.
 \Bukovsky's work was motivated by the following question which \Vopenka\, asked in 1966:
Can one change the cofinality of a regular cardinal without collapsing smaller cardinals (see \cite{Bukovsky75})?
Prikry solved \Vopenka's question for measurable cardinals in his dissertation \cite{Prikry70}.
The work of \Bukovsky\ and of Namba solved \Vopenka's question for $\aleph_2$, which is now known, due to Jensen's covering theorem, to be the only possibility without assuming large cardinals.

In the late 1960's, Prikry solved Solovay's question for the case when $\kappa$ has cofinality $\om$
 and  $\mu^{\om}<\kappa$ for all $\mu<\kappa$.
His proof was never published, but his result is quoted in \cite{Namba70}.
In this article, we provide 
modified versions of Prikry's original proofs, generalizing them to cardinals of uncountable cofinality whenever this is straightforward. 
The perfect tree forcings constructed by Prikry are interesting  in their own right, and his original results provided the impetus for the recent results in this article, further  investigating their forcing properties.

\Bukovsky\ and \Coplakova\  conducted a comprehensive study of 
 forcing properties of generalized Namba  forcing and of a family of perfect tree forcings
 in  \cite{Bukovsky/Coplakova90}.
 They found which distributive laws hold, which cardinals are collapsed, and proved under certain assumptions that the forcing extensions  are minimal for adding new $\om$-sequences.
Their perfect tree forcings,
 defined in Section 3 of \cite{Bukovsky/Coplakova90},
are similar, but not equivalent, to the forcings investigated in this paper;
some of their techniques are appropriated in later sections.
A  variant of Namba style tree forcings, augmented from Namba forcing analogously to how the perfect tree forcings in \cite{Bukovsky/Coplakova90} are augmented from those in \cite{Bukovsky75}, was used 
by Cummings, Foreman and Magidor in \cite{Cummings/Foreman/Magidor03} to prove that a supercompact cardinal can be forced to collapse to $\aleph_2$ so that in this forcing extension, 
$\square_{\om_n}$ holds for all positive integers $n$, and each stationary subset of $\aleph_{\om+1}\cap \mathrm{cof}(\om)$ reflects to an $\al$ with cofinality $\om_1$.
We point out that the addition of a new $\om$-sequence of ordinals has consequences for the co-stationarity 
of the ground model in the $\mathcal{P}_{\mu}(\lambda)$ of the extension model.
It follows from more general work in \cite{DobrinenOmSeq06}  that if the ground model $V$ satisfies $\square_{\mu}$ for all regular cardinals $\mu$ in  forcing extension $V[G]$ and if
$V[G]$ contains a new sequence $f:\om\ra\kappa$, then for all cardinals $\mu<\lambda$
in $V[G]$ with $\mu$ regular in $V[G]$ and $\lambda\ge \kappa$,
 $(\mathcal{P}_{\mu}(\lambda))^{V[G]}\setminus V$ is stationary in $(\mathcal{P}_{\mu}(\lambda))^{V[G]}$.
It seems likely that further investigations of variants of Namba and perfect tree forcings should lead to interesting results.

A complete Boolean algebra $\mathbb{B}$ is said to {\em satisfy the $(\lambda,\mu)$-distributive law}
($(\lambda,\mu)$-d.l.)
 if for each collection of $\lambda$ many partitions of unity into at most $\mu$ pieces, there is a common refinement. 
This is equivalent to saying that forcing with  $\mathbb{B}\setminus \{\mathbf{0}\}$ does not add any new functions from $\lambda$ into $\mu$.
The weaker three-parameter distributivity is defined as follows:
$\mathbb{B}$  {\em satisfies the $(\lambda,\mu,<\delta)$-distributive law} 
($(\lambda,\mu,<\delta)$-d.l.)
if in any  forcing extension $V[G]$ by  $\mathbb{B}\setminus\{\mathbf{0}\}$,
for each function $f:\lambda\ra\mu$ in $V[G]$, there is a function $h:\lambda\ra[\mu]^{<\delta}$ in the ground model $V$ such that $f(\al)\in h(\al)$, for each $\al<\lambda$.
Such a function $h$ may be thought of as a covering of $f$ in the ground model.  
Note that  the $\delta$-chain condition implies 
 $(\lambda,\mu,<\delta)$-distributivity, for all $\lambda$ and $\kappa$.
We
 shall usually write $(\lambda, \mu,\delta)$-distributivity instead of   $(\lambda,\mu,<\delta^+)$-distributivity.
See \cite{KoppelbergHB} for more background on distributive laws.

In this paper,
given any strictly increasing sequence of regular cardinals $\lgl \kappa_n :n<\om\rgl$,  letting $\kappa=\sup_{n<\om}\kappa_n$ and assuming that $\mu^{\om}<\kappa$ for all $\mu<\kappa$,
 $\mathbb{P}$  is a collection of certain perfect subtrees of $\prod_{n<\om}\kappa_n$, partially ordered by inclusion, described in
 Definition \ref{defn.2.5}.
Let $\mathbb{B}$ denote its Boolean completion.
We  prove the following.
$\mbb{P}$ has size $\kappa^\omega$
 and $\mbb{B}$ has maximal antichains of
 size $\kappa^\omega$, but no larger.
$\mbb{P}$ satisfies the
 $(\omega, \kappa_n)$-d.l.\ for each $n < \omega$
 but not the
 $(\omega, \kappa)$-d.l.
In fact, it does not satisfy the
 $(\omega, \kappa, \kappa_n)$-d.l.\ for any $n < \omega$.
It does, however, satisfy the
 $(\omega, \kappa, {<\kappa})$-d.l., and in fact it satisfies
 the $(\omega, \infty, {<\kappa})$-d.l.,
 because it satisfies a Sacks-like property.
On the other hand,
 the $(\mf{d}, \infty, {<\kappa})$-d.l.\ fails.
We do not know if $\infty$ can be replaced by a  cardinal strictly smaller than
 $\kappa^\omega$.
However, we do know that the
 $(\mf{h}, 2)$-d.l.\ fails.
 ($\mf{h}$ and $\mf{d}$ are cardinal characteristics of the
 continuum, and $\omega_1 \le \mf{h} \le \mf{d} \le 2^\omega$.)
In fact, we have that $P(\omega)/\Fin$ densely embeds
 into the regular open completion of $\mbb{P}$.
By similar reasoning, we show that forcing with
 $\mbb{P}$ collapses $\kappa^\omega$ to $\mf{h}$.
Under the assumption that $\kappa$ is the limit of measurables,
 we have that
 every $\omega$-sequence of ordinals in the extension is either
 in the ground model or it constructs the generic filter.
If $G$ is $\mbb{P}$-generic over $V$ and
 $H \in V[G]$ is $P(\omega)/\Fin$-generic
 over $V$, then since $P(\omega)/\Fin$ does not
 add $\omega$-sequences, $G \not\in V[H]$.
Thus, $\mbb{P}$
 does not add a minimal degree of constructibility.
Some of the results also hold for  cardinals $\kappa$ of uncountable cofinality, and these are presented in full generality.
The article closes with an example of what can go wrong when $\kappa$ has uncountable cofinality, 
highlighting some open problems and ideas for how to approach them.

\section{Definitions and Basic Lemmas}\label{sec.defs}

\subsection{Basic Definitions}

Recall that given a separative poset $\mbb{P}$,
 the \textit{regular open completion} $\mbb{B}$ of $\mbb{P}$
 is a complete Boolean algebra into which $\mbb{P}$
 densely embeds (after we remove the zero
 element $\mathbf{0}$  from $\mbb{B}$).
Every other such complete
 Boolean algebra is isomorphic to $\mbb{B}$.
A set $C \subseteq \mbb{P}$ is
 \textit{regular open} iff
\begin{itemize}
\item[1)]
 $(\forall p_1 \in C)(\forall p_2 \le p_1)\,
 p_2 \in C$, and
\item[2)] $(\forall p_1 \not\in C)
 (\exists p_2 \le p_1)
 (\forall p_3 \le p_2)\,
 p_3 \not\in C$.
\end{itemize}
Topologically, giving $\mathbb{P}$ the topology generated by basic open sets of the form 
$\{q\in\bP : q\le p\}$ for $p\in \bP$, a set 
$C\subseteq \mbb{P}$ is regular open if and only if it is equal to the interior of its closure in this topology.
We define $\mbb{B}$ as the collection
 of regular open subsets of $\mbb{P}$ ordered by inclusion.
See \cite{KoppelbergHB} for more background on the regular open completion of a partial ordering.

Given cardinals $\lambda$ and $\mu$,
 we say $\mbb{B}$ (or $\mbb{P}$) satisfies the
 $(\lambda, \mu)$-distributive law ($(\lambda, \mu)$-d.l.)
 if and only if  whenever $\{ A_\alpha : \alpha < \lambda \}$
 is a collection of size $\le \mu$ maximal
 antichains in $\mbb{B}$,
 there is a single $p \in \mbb{B}$
 below one element of each antichain.
This is equivalent to the statement
 $1_\mbb{B} \forces (^{\check{\lambda}} \check{\mu}
 \subseteq \check{V})$.
That is, every function from $\lambda$ to $\mu$
 in the forcing extension is already in the ground model.
Note that $\mbb{B}$ and $\mbb{P}$ force the same statements, 
since 
  $\mbb{P}$ densely embeds into $\mbb{B}$
 by the mapping $p \mapsto \{ q \in \mbb{P} : q \le p \}$.
The $(\lambda, \mu)$-d.l.\
 is equivalent to the statement that
 whenever $p \in \mbb{P}$ and $\dot{f}$ are such that
 $p \forces \dot{f} : \check{\lambda} \to \check{\kappa}$,
 then there are $q \le p$ and $g : \lambda \to \kappa$
 satisfying $q \forces \dot{f} = \check{g}$.
We will also study
 a distributive law weaker than the $(\lambda,\mu)$-d.l.; 
namely, the
 $(\lambda, \mu, {<\delta})$-d.l.
  where  $\delta \le \mu$.
This is the statement that
 for each $\alpha < \lambda$ there is a set
 $X_\alpha \in [A_\alpha]^{<\delta}$ such that
 there is a single non-zero element of $\mbb{B}$
 below $\bigvee X_\alpha$ for each $\alpha < \lambda$.
That is, there is some $p \in \mbb{P}$ such that
 $(\forall \alpha < \lambda)
  (\exists a \in X_\alpha)\,
  p \in a$.
The $(\lambda, \mu, {<\delta})$-d.l.\
 is equivalent to the statement that whenever
 $p \in \mbb{P}$ and $\dot{f}$ satisfy
 $p \forces \dot{f} : \check{\lambda} \to \check{\mu}$,
 then there exists $q \le p$ and a function
 $g : \lambda \to [\mu]^{<\delta}$ satisfying
 $q \forces (\forall \alpha < \check{\lambda})\,
 \dot{f}(\alpha) \in \check{g}(\alpha)$.
Finally,
 if $\mu$ is the smallest cardinal such that
 every maximal antichain in $\mbb{B}$ has size
 $\le \mu$, then the distributive law is unchanged
 if we replace $\mu$ in the second argument
 with any larger cardinal,
 so in this situation we write $\infty$
 instead of $\mu$.

\begin{convention}\label{conventionforpaper}
For this entire paper,
 $\kappa$ is a singular cardinal
 and $\langle \kappa_\alpha :
 \alpha < \cf(\kappa) \rangle$
 is an increasing
 sequence of regular cardinals
 with limit $\kappa$ such that
 $\cf(\kappa) < \kappa_\alpha < \kappa$
 for all $\alpha$.
\end{convention}

Note that
the cardinality of 
 $\prod_{\alpha < \cf(\kappa)}
 \kappa_\alpha $ equals  $ \kappa^{\cf(\kappa)}$, which is greater than $\kappa$.
We do not assume that $\kappa$
 is a strong limit cardinal.
However, we do make the following
 weaker assumption:

\begin{assumption}\label{basicassumption}
 $$(\forall \mu < \kappa)\,
 \mu^{\cf(\kappa)} < \kappa.$$
\end{assumption}
In a few places, we will make the
 special assumption that $\kappa$
 is the limit of measurable cardinals.

\begin{definition}\label{def.2.1}
The set $N \subseteq {^{<\cf(\kappa)} \kappa}$
 consists of all functions $t$ such that
 $\dom(t) < \cf(\kappa)$ and
 $(\forall \alpha \in \dom(t))\,
 t(\alpha) < \kappa_\alpha$.
We call each $t \in N$ a \defemph{node}.
Given a set $T \subseteq N$
 (which is usually a tree,
 meaning that it is closed under initial segments),
 $[T]$ is the set of all $f \in
 {^{\cf(\kappa)} \kappa}$ such that
 $(\forall \alpha < \cf(\kappa))\,
 f \restriction \alpha \in T$.
Define $X := [N]$.
Given $t_1, t_2 \in N \cup X$,
 we write $t_2 \sqsupseteq t_1$ iff
 $t_2$ is an extension of $t_1$.
\end{definition}

Note that $|N| = \kappa$
 and $|X| = \kappa^{\cf(\kappa)}$.
We point out that our set $X$ is commonly written as $\prod_{\al<\cf(\kappa)}\kappa_{\al}$.
In order to avoid confusion with cardinal arithmetic and to simplify notation, we shall use $X$ as defined above.

\begin{definition}\label{def.2.2}
Fix a tree $T \subseteq N$.
A \defemph{branch} through $T$
 is a maximal element of $T \cup [T]$.
Given $\alpha < \cf(\kappa)$,
 $T(\alpha) := T \cap {^\alpha \kappa}$
 is the set of all nodes of $T$ on
 \defemph{level} $\alpha$.
Given $t \in T$ such that $t \in T(\alpha)$,
 then $\mathrm{Succ}_T(t)$ is the set
 of all children of $t$ in $T$:
 all nodes $c \sqsupseteq t$ in $T(\alpha+1)$.
The word \defemph{successor} is another word
 for child (hence, successor
 always means immediate successor).
A node $t \in T$ is \defemph{splitting}
 iff $|\mathrm{Succ}_T(t)| > 1$.
$\mathrm{Stem}(T)$
 is the unique (if it exists)
 splitting node of $T$
 that is comparable
 (with respect to extension)
 to all other elements of $T$.
Given $t \in T$,
 the tree $T | t$ is the subset of $T$
 consisting of all nodes of $T$
 that are comparable to $t$.
\end{definition}

It is desirable for the trees
 that we consider to have no dead ends.

\begin{definition}\label{def.2.3}
A tree $T \subseteq N$ is called
 \defemph{non-stopping} iff
 it is non-empty
 and for every
 $t \in T$, there is some $f \in [T]$
 satisfying $f \sqsupseteq t$.
A tree $T \subseteq N$ is \defemph{suitable}
 iff $T$ has no branches of length
 $< \cf(\kappa)$.
\end{definition}

Suitable implies non-stopping,
 and they are equivalent if
 $\cf(\kappa) = \omega$.

\begin{definition}\label{defn.2.4}
A tree $T \subseteq N$ is
 \defemph{pre-perfect} iff $T$
 is non-stopping and
 for each $\alpha < \cf(\kappa)$ and
 each node $t_1 \in T$,
 there is some $t_2 \sqsupseteq t_1$ in $T$
 such that $|\mathrm{Succ}_T(t_2)| \ge \kappa_\alpha$.
A tree $T \subseteq N$ is \defemph{perfect}
 iff $T$ is pre-perfect and,
 instead of just being non-stopping, is suitable.
\end{definition}

In Section~\ref{secuncheight},
 we will construct a pre-perfect $T$
 such that $[T]$ has size $\kappa$.
That example points out problems that arise in straightforward attempts to generalize some of our results to singular cardinals of uncountable cofinality.
On the other hand, it is not hard to see that
 if $T$ is perfect, then $[T]$ has size $\kappa^{\cf(\kappa)}$.
We will now define the forcing that we will investigate.

\begin{definition}\label{defn.2.5}
$\mbb{P}$ is the set of all perfect trees $T \subseteq N$
 ordered by inclusion.
$\mbb{B}$ is the regular open completion of $\mbb{P}$.
\end{definition}

Note that by a density argument, given $\kappa$, the choice of the sequence
 $\langle \kappa_\alpha : \alpha < \cf(\kappa) \rangle$ having $\kappa$ as its limit
 does not affect the definition of $\mbb{P}$.

\begin{definition}\label{defn.2.6}
Assume $\cf(\kappa) = \omega$.
Fix a perfect tree $T \subseteq N$.
A node $t \in T$ is $0$-splitting iff
 it has exactly $\kappa_0$ children in $T$ and
 it is the stem of $T$ (so it is unique).
Given $n < \omega$,
 a node $t \in T$ is $(n+1)$-splitting iff
 it has exactly $\kappa_{n+1}$ children in $T$ and
 it's maximal proper initial segment that is splitting
 is $n$-splitting.
\end{definition}

\begin{definition}\label{defn.2.7}
Assume $\cf(\kappa) = \omega$.
Fix a perfect tree $T \subseteq N$.
We say $T$ is in \defemph{weak splitting normal form}
 iff every splitting node of $T$ is $n$-splitting
 for some $n$.
We say $T$ is in \defemph{medium splitting normal form}
 iff it is in weak splitting normal form
 and for each splitting node $t \in T$,
 all minimal splitting descendants of $t$
 are on the same level.
We say $T$ is in \defemph{strong splitting normal form}
 iff it is in medium splitting normal form and
 for each $n \in \omega$, there is some $l_n \in \omega$
 such that $T(l_n)$ is precisely the set of
 $n$-splitting nodes of $T$.
We say that the set
 $\{ l_n : n \in \omega \}$
 \defemph{witnesses}
 that $T$ is in strong splitting normal form.
\end{definition}

If $T$ is in weak splitting normal form,
 then for each $f \in [T]$,
 there is a sequence
 $t_0 \sqsubseteq t_1 \sqsubseteq ...$
 of initial segments of $f$ such that
 $t_n$ is $n$-splitting for each $n < \omega$
 (and these are the only splitting nodes on $f$).
It is not hard to prove that any
 $T \in \mbb{P}$ can be extended to some
 $T' \le T$ in medium splitting normal form.
Furthermore, the
 set of conditions below a condition in
 medium splitting normal form
 is isomorphic to $\mbb{P}$ itself.
This implies that whenever $\varphi$
 is a sentence in the forcing language
 that only involves names of the form
 $\check{a}$ for some $a \in V$, then either
 $1 \forces \varphi$ or $1 \forces \neg \varphi$.
In Proposition~\ref{cangetstrongnormalform},
 we will show
 (in the $\cf(\kappa) = \omega$ case)
 that each condition can be
 extended to one in \textit{strong} splitting normal form.


\subsection{Topology}

To prove several facts about $\mbb{P}$
 for the $\cf(\kappa) = \omega$ case,
a topological approach will be useful.

\begin{definition}
Given $t \in N$,
 let $B_t \subseteq X$ be the set of all
 $f \in X$ such that $f \sqsupseteq t$.
We give the set $X$ the topology
 induced by the basis
 $\{ B_t : t \in N \}$.
\end{definition}

\begin{observation}
Each $B_t \subseteq X$ for $t \in N$
 is clopen.
\end{observation}

\begin{observation}
A set $C \subseteq X$ is closed iff
 whenever $g \in X$ satisfies
 $(\forall \alpha < \cf(\kappa))\,
 |C \cap B_{g \restriction \alpha}|
 \not= \emptyset$,
 then $g \in C$.
\end{observation}

This next fact explains
 why we considered the concept of ``non-stopping'':

\begin{fact}
A set $C \subseteq X$ is closed iff
 $C = [T]$ for some (unique) non-stopping
 tree $T \subseteq N$.
\end{fact}

\begin{definition}
A set $C \subseteq X$ is \defemph{strongly closed}
 iff $C = [T]$ for some (unique)
 suitable tree $T \subseteq N$.
Hence, if $\cf(\kappa) = \omega$,
 then strongly closed is the same as closed.
\end{definition}

\begin{definition}
A set $P \subseteq X$ is
 \defemph{perfect} iff 
 it is strongly closed and for each
 $f \in P$,
 every neighborhood of $f$
 contains $\kappa^{\cf(\kappa)}$
 elements of $P$.
\end{definition}

Thus, every non-empty perfect set
 has size $\kappa^{\cf(\kappa)} = |X|$.
One can check that if $B \subseteq X$
 is clopen and
 $P \subseteq X$ is perfect, then
 $B \cap P$ is perfect.
The next lemma does not hold
 in the $\cf(\kappa) > \omega$ case
 when we replace ``perfect tree''
 with ``pre-perfect tree'',
 because it is possible for a pre-perfect
 tree to have $\kappa$ branches
 (see Counterexample~\ref{uncountablecounterex}).

\begin{lemma}
\label{perfecttreeimpliesset}
If $T \subseteq N$ is a perfect tree, then
 $[T]$ is a perfect set.
\end{lemma}
\begin{proof}
Since $T$ is perfect,
 it is suitable, which by definition implies
 that $[T]$ is strongly closed.
Next, given any $t \in T$,
 we can argue that $B_t \cap [T]$ has size $\kappa^{\cf(\kappa)}$,
 because we can easily construct an embedding from
 $N$ into $T | t$, and we have that $X$ has size
 $\kappa^{\cf(\kappa)}$.
\end{proof}

This next lemma implies the opposite direction:
 if $P \subseteq X$ is a perfect set,
 then $P = [T]$ for some perfect tree $T \subseteq N$.

\begin{lemma}
\label{closedandeverywherebigimpliesperfect}
Fix $P \subseteq X$.
Suppose $P$ is strongly closed and
 for each $f \in P$,
 every neighborhood of $f$
 contains $\ge \kappa$
 elements of $P$.
Then $P = [T]$ for some (unique) perfect tree
 $T \subseteq N$.
Hence, $P$ is a perfect set.
\end{lemma}
\begin{proof}
Since $P$ is strongly closed,
 fix some (unique) suitable tree
 $T \subseteq N$ such that $P = [T]$.
If we can show that $T$ is a perfect tree,
 we will be done
 by the lemma above.

Suppose that $T$ is not a perfect tree.
Let $t \in T$ and $\alpha < \cf(\kappa)$
 be such that for every extension $t' \in T$
 of $t$, $| \mathrm{Succ}_T(t') | \le \kappa_\alpha$.
We see that $[(T|t)]$ has size at most
 $(\kappa_\alpha)^{\cf(\kappa)} < \kappa$,
 which is a contradiction.
\end{proof}


\begin{cor}
\label{majorequiv}
Fix $P \subseteq X$.
The following are equivalent:
\begin{itemize}
\item[1)] $P$ is perfect;
\item[2)] $P$ is strongly closed and
 $$(\forall f \in P)
  (\forall \alpha < \cf(\kappa))\,
  |P \cap B_{f \restriction \alpha}|
  = \kappa^{\cf(\kappa)};$$
\item[3)] $P$ is strongly closed and
 $$(\forall f \in P)
  (\forall \alpha < \cf(\kappa))\,
  |P \cap B_{f \restriction \alpha}|
  \ge \kappa;$$
\item[4)] There is a perfect
 tree $T \subseteq N$ such that
 $P = [T]$.
\end{itemize}
\end{cor}

\begin{lemma}
\label{bigclosedhasperfect}
Assume $\cf(\kappa) = \omega$.
Let $C \subseteq X$ be strongly closed
 and assume $|C| > \kappa$.
Then $C$ has a non-empty perfect subset.
\end{lemma}

\begin{proof}
Let $T \subseteq N$ be the (unique) suitable tree
 such that $C = [T]$.
We will construct $T'$ by successively adding
 elements to it, starting with the empty set.
By an argument similar
 to the one used in the previous lemma,
 there must be a node $t_\emptyset \in T$
 such that there is a set
 $S_{t_\emptyset} \subseteq \mathrm{Succ}_T(t_\emptyset)$
 of size $\kappa_0$ such that
 $(\forall c \in S_{t_\emptyset})\,
 [(T|c)] > \kappa$.
Fix $t_\emptyset$ and add it and all its initial segments
 to $T'$.
Next, for each $c \in S_{t_\emptyset}$,
 there must be a node $t_c \in T$
 such that there is a set
 $S_{t_c} \subseteq \mathrm{Succ}_T(t_c)$
 of size $\kappa_1$ such that
 $(\forall d \in S_{t_c})$\,
 $[(T|d)] > \kappa$.
For each $c$,
 fix such a $t_c$
 and add it and all its initial segments to $T'$.
Continue like this.
At a limit stage $\alpha$,
 let $t$ be such that it is not in $T'$
 but all its initial segments are in $T'$.
Find some extension of $t$ in $T$ that has
 $\kappa_\alpha$ appropriate children, etc.
It is clear from the construction that
 $T' \subseteq T$ will be a perfect tree.
\end{proof}




\subsection{Laver-style  Trees}

In this subsection, we assume $\cf(\kappa) = \omega$,
 as this is the only case to which
 the proofs apply.
The results in this subsection are modifications to our setting  of work extracted from \cite{Namba70}, where Namba used the terminology `rich' and `poor' sets.

\begin{definition}
For each $n < \omega$, let
 $\mbb{Q}_n \subseteq \mbb{P}$
 denote the set of $T \in \mbb{P}$ such that
 $\dom( \mathrm{Stem}(T) ) \le n$, and
 for each $m \ge \dom( \mathrm{Stem}(T) )  $
 and $t \in T(m)$,
 $| \mathrm{Succ}_T(t) | = \kappa_m$.
\end{definition}

Note that if $n < m$, then
 $\mbb{Q}_n \subseteq \mbb{Q}_m$.
The set $\mbb{Q}=\bigcup_{n < \omega} \mbb{Q}_n$
 is the collection of ``Laver'' trees.

\begin{definition}
Fix a tree $T \subseteq N$.
We say that $T$
 \defemph{has small splitting at level} $n < \omega$
 iff $(\forall t \in T(n))\,
 | \mathrm{Succ}_T(t) | < \kappa_n$.
A tree is called {\em leafless} if it has no maximal nodes.
We say that $T$ is $n$-\defemph{small} iff
 there is a sequence of leafless trees
 $\langle D_m \subseteq N : m \ge n \rangle$
 such that $[T] \subseteq \bigcup_{m \ge n} [D_m]$
 and each $D_m$ has small splitting at level $m$.
\end{definition}

Note that if $n > m$,
 then $n$-small implies $m$-small.
If $\langle D_m : m \ge n \rangle$
 witnesses that $T$ is $n$-small,
 then without loss of generality
 $D_m \subseteq T$ for all $m \ge n$.

\begin{observation}
Let $m < \omega$.
Let $\mc{D}$ be a collection of trees
 that have small splitting at level $m$.
If $|\mc{D}| < \kappa_m$,
 then $\bigcup \mc{D}$ has small splitting
 at level $m$.
\end{observation}

\begin{lemma}\label{lem2.23}
Let $T \subseteq N$ be a tree, let
 $t := \mathrm{Stem}(T)$,
 and let $n := \dom(t)$.
Assume that $T$ is not $n$-small.
Then
 $$E := \{ c \in \mathrm{Succ}_T(t) :
 (T | c) \mbox{ is not $(n+1)$-small} \}$$
 has size $\kappa_n$.
\end{lemma}
\begin{proof}
Towards a contradiction,
 suppose that $|E| < \kappa_n$.
Let $F := \mathrm{Succ}_T(t) - E$.
Let $D_n \subseteq N$ be the set
 $D_n := \bigcup \{ (T|c) : c \in E \}$.
Note that
 $$T = [D_n] \cup
  \bigcup_{c \in F} [T|c].$$
We have that $D_n$ has small splitting
 at level $n$, because $t$ is the only
 node in $D_n \subseteq T$ at level $n$,
 and $\mathrm{Succ}_{D_n}(t) = E$
 has size $< \kappa_n$.

For each $c \in F$,
 let $\langle D^c_m \subseteq (T|c) :
 m \ge n + 1 \rangle$
 be a sequence of trees that witnesses that
 $(T|c)$ is $(n+1)$-small.
For each $m \ge n+1$, let
 $$D_m := \bigcup_{c \in F} D^c_m.$$
Then 
 $$\bigcup_{c \in F} [T|c] =
 \bigcup_{c \in F} \bigcup_{m \ge n+1} [D^c_m] =
 \bigcup_{m \ge n+1} \bigcup_{c \in F} [D^c_m] \subseteq
 \bigcup_{m \ge n+1} [D_m].$$

Consider any $m \ge n+1$.
Since $|F| \le |\mathrm{Succ}_T(t)|
 \le \kappa_n < \kappa_m$
 and each $D^c_m$ has small splitting at level $m$,
 by the observation above
 $D_m$ has small splitting at level $m$.
Thus, we have
 $[T] \subseteq \bigcup_{m \ge n} [D_m]$ and
 each $D_m$ has small splitting at level $m$.
Hence $T$ is $n$-small,
 which is a contradiction.
\end{proof}

\begin{cor}
\label{notsmallimplieslaver}
Let $T \subseteq N$ be a tree,
 let $t := \mathrm{Stem}(T)$, and
 let $n := \dom(t)$.
Assume that $T$ is not $n$-small.
Then there is a subtree $L \subseteq T$
 such that $L \in \mbb{Q}_n$.
\end{cor}

\begin{proof}
We will construct $L$ by induction.
For each $m \le n$,
 let $L(m) := \{ t \restriction m \}$.
Let $L(n+1)$ be the set of $c \in \mathrm{Succ}_T(t)$
 such that $(T|c)$ is not $(n+1)$-small.
By Lemma \ref{lem2.23},
 $|\mathrm{Succ}_T(t)| = \kappa_n$.
Let $L(n+2)$ be the set of nodes of the form
 $c \in \mathrm{Succ}_T(u)$ for $u \in L(n+1)$
 such that $(T|c)$ is not $(n+2)$-small.
Again by Lemma  \ref{lem2.23},
 for each $u \in L(n+1)$,
 since $(T|u)$ is not $(n+1)$-small,
 $|\mathrm{Succ}_L(u)| = \kappa_{n+1}$.
Continuing in this manner,
 we obtain $L \subseteq T$, and it has the property that
 for each $m \ge n$ and $t \in L(m)$,
 $|\mathrm{Succ}_L(t)| = \kappa_m$.
Thus, $L \in \mbb{Q}_n$.
\end{proof}

\begin{lemma}\label{lem.2.25}
Fix $n < \omega$ and
 let $L \in \mbb{Q}_n$.
Then $L$ is not $n$-small.
\end{lemma}

\begin{proof}
Suppose, towards a contradiction,
 that there is a sequence of leafless trees
 $\langle D_m \subseteq L : m \ge n \rangle$
 such that $[L] \subseteq \bigcup_{m \ge n} [D_m]$
 and each $D_m$ has small splitting at level $m$.
Let $t_n \in L(n)$ be arbitrary.
We will define a sequence of nodes
 $\langle t_m \in L(m) : m \ge n \rangle$
 such that $t_n \sqsubseteq t_{n+1} \sqsubseteq ...$
 and $(\forall m \ge n)\,
 [D_m] \cap B_{t_{m+1}} = \emptyset$.
If we let $x \in [L]$ be
the union of this sequence of $t_n$'s,
 then since $\{x\} = \bigcap_{m \ge n} B_{t_{m+1}}$,
 we will have $x \not\in \bigcup_{m \ge n} [D_m]$,
 so $[L] \not\subseteq \bigcup_{m \ge n} [D_m]$,
 which is a contradiction.

Define $t_{n+1}$ to be any successor of $t_n$ in $L$
 such that $t_{n+1} \not\in D_n$.
This is possible because $D_n$ has small splitting
 at level $n$ and $t$ has $\kappa_n$ successors in $L$.
We have $[D_n] \cap B_{t_{n+1}} = \emptyset$.
Next, define $t_{n+2}$ to be any successor of $t_{n+1}$ in $L$
 such that $t_{n+2} \not\in D_{n+1}$.
Continuing in this manner yields 
 the desired sequence
 $\langle t_m : m\ge n \rangle$.
\end{proof}

\begin{proposition}\label{prop.clear}
Fix $n < \omega$.
If $\mc{T}$ is a collection of $n$-small trees
 and $|\mc{T}| < \kappa_n$, then
 $\bigcup \mc{T}$ is an $n$-small tree.
\end{proposition}

\begin{proof}
For each $T \in \mc{T}$,
 let $\langle D^T_m : m \ge n \rangle$ witness that $T$
 is $n$-small.
Then $\langle \bigcup_{T \in \mc{T}} D^T_m : m \ge n \rangle$
 witnesses that $\bigcup \mc{T}$ is $n$-small.
\end{proof}

\begin{cor}\label{cor.2.27}
Fix $n < \omega$.
If $\{ [T] : T \in \mc{T} \}$ is a partition of $X$
 into $< \kappa_n$ closed sets, then at least one
 of the trees $T \in \mc{T}$ is not $n$-small.
\end{cor}

\begin{proof}
Suppose that each $T \in \mc{T}$ is $n$-small.
Then by 
Proposition \ref{prop.clear},
 $\bigcup_{T \in \mc{T}} T = N$ is $n$-small.
However, $N$ cannot be $n$-small by 
Lemma \ref{lem.2.25},
as $N$ is a member of  $ \mbb{Q}_n$.
\end{proof}

We do not know if this next lemma
 has an analogue
 for the $\cf(\kappa) > \omega$ case because
 of a Bernstein set phenomenon.

\begin{lemma}\label{lem.psi}
\label{onlysomanypieces}
Assume $\cf(\kappa) = \omega$.
Fix $n < \omega$.
Suppose $\Psi: N \to \kappa_n$.
Given $h : \omega \to \kappa_n$,
 let $C_h \subseteq X$ be the set of all
 $f \in X$ such that
 $$(\forall k < \omega)\,
 \Psi( f \restriction k ) = h(k).$$
Then for some $h$,
 there is an $L \in \mbb{Q}_m$  
 such that $[L] \subseteq C_h$, where $m$ satisfies $\kappa_m>(\kappa_n)^{\om}$.
\end{lemma}

\begin{proof}
It is straightforward to see that each
 set $C_h$ is strongly closed (and hence closed).
Let $m < \omega$ be such that 
$(\kappa_n)^{\om}<\kappa_m$.
 Such an $m$ exists 
by Assumption \ref{basicassumption}.
By Corollary \ref{cor.2.27},
 one of the sets $C_h = [T]$ must be such that
 $T$ is not $m$-small.
By Corollary~\ref{notsmallimplieslaver},
 there is some tree $L \subseteq T$
 such that $L \in \mbb{Q}_m$.
\end{proof}



\subsection{Strong Splitting Normal Form}

\begin{observation}
\label{embeddingobs}
Let $T \in \mbb{P}$.
There is an embedding
 $F : N \to T$, meaning that
 $(\forall t_1, t_2 \in N)$,
\begin{itemize}
\item $t_1 = t_2 \Leftrightarrow
 F(t_1) = F(t_2)$;
\item $t_1 \sqsubseteq t_2 \Leftrightarrow
 F(t_1) \sqsubseteq F(t_2)$;
\item $t_1 \perp t_2 \Leftrightarrow
 F(t_1) \perp F(t_2)$.
\end{itemize}
From this, it follows by induction that if
 $t \in N$ is on level $\alpha < \cf(\kappa)$,
 then $F(t)$ is on level $\beta$
 for some $\beta \ge \alpha$.
It follows that given any $f \in [N]$,
 there is exactly one $g \in [T]$
 that has all the nodes
 $F(f \restriction \alpha)$
 for $\alpha < \cf(\kappa)$ as initial segments.

Given a set $S \subseteq N$,
 let $I(S)$ be the set of all initial
 segments of elements of $S$.
If $H \subseteq N$ is a perfect tree, then
 $I(F``(H)) \subseteq T$ is a perfect tree.
If $H_1, H_2 \subseteq N$ are trees such that
 $[H_1] \cap [H_2] = \emptyset$, then
 $[I(F``(H_1))] \cap [I(F``(H_2))] = \emptyset$.
\end{observation}

\begin{proof}
To construct the embedding $F$,
 first define $F(\emptyset) = \emptyset$.
Now fix $\alpha < \cf(\kappa)$ and suppose
 $F(u)$ has been defined for all
 $u \in \bigcup_{\gamma < \alpha} N(\gamma)$.
If $\alpha$ is a limit ordinal
 and $t \in N(\alpha)$,
 define $F(t)$ to be
 $\bigcup_{\gamma < \alpha} F(t \restriction \gamma)$.
If $\alpha = \beta + 1$,
 fix $u \in N(\beta)$.
Fix $s \sqsupseteq F(u)$
 such that $s$ has $\ge \kappa_\beta$ successors in $T$.
For each $\sigma < \kappa_\beta$,
 define $F(u ^\frown \sigma)$ to be the
 $\sigma$-th successor of $s$ in $T$.
The rest of the claims in the observation
 follow easily.
\end{proof}

\begin{proposition}
\label{cangetstrongnormalform}
For each $T \in \mbb{P}$,
 there is some $T' \le T$
 in strong splitting normal form.
\end{proposition}
\begin{proof}
Fix $T \in \mbb{P}$.
Fix an embedding $F : N \to T$.
Let $\Psi : N \to \omega$ be the coloring
 $\Psi(u) := \dom( F(u) )$.
Let $L \in  \mbb{Q}$
 be given by Lemma \ref{lem.psi}.
Then $T' := I(F``(L))$ is in strong
 splitting normal form and $T' \le T.$
\end{proof}

This section concludes by showing  that $\mbb{P}$ is not
 $\kappa^{\cf(\kappa)}$-c.c.
That is,
 $\mbb{P}$ has a maximal antichain
 of size $\kappa^{\cf(\kappa)}$.
This result is optimal because
 $|\mbb{P}| = \kappa^{\cf(\kappa)}$.

\begin{proposition}\label{prop.2.31}
\label{splititup}
Let $T \in \mbb{P}$.
Then there are $\kappa^{\cf(\kappa)}$
 pairwise incompatible extensions of $T$ in $\mbb{P}$.
Hence, $\mbb{P}$ is not $\kappa^{\cf(\kappa)}$-c.c.
\end{proposition}

\begin{proof}
Let $F : N \to T$ be an embedding
 guaranteed to exist by the observation above.
For each $\alpha < \cf(\kappa)$, let
 $\{ R_{n, \beta} : \beta < \kappa_\alpha \}$
 be a partition of $\kappa_\alpha$
 into $\kappa_\alpha$
 pieces of size $\kappa_\alpha$.
Given $f \in [N]$, let
 $H_f \subseteq N$ be the tree
 $$H_f := \{ t \in N :
 (\forall \alpha \in \dom(t))\,
 t(\alpha) \in R_{\alpha, f(\alpha)} \}.$$
Each $H_f$ is a non-empty perfect tree.
If $f_1 \not= f_2$, then
 $[H_{f_1}] \cap [H_{f_2}] = \emptyset.$
Using the notation of Proposition~\ref{embeddingobs},
 for each $f \in [N]$ let
 $$T_f := I(F``(H_f)).$$
Certainly each $[T_f]$ is a subset of $P$,
 because $T_f \subseteq T$.
By the Proposition~\ref{embeddingobs},
 each $T_f$ is a non-empty perfect tree,
 and $f_1 \not= f_2$ implies
 $[T_{f_1}] \cap [T_{f_2}] = \emptyset$,
 which in turn implies
 $T_{f_1}$ is incompatible with $T_{f_2}$.
Thus, the conditions $T_f \in \mbb{P}$ for $f \in [N]$
 are pairwise incompatible.
Since $[N] = X$ has size $\kappa^{\omega}$,
 there are $\kappa^{\omega}$
 of these conditions.
\end{proof}


\section{$(\om,\kappa_n)$ and $(\om,\infty,<\kappa)$-distributivity hold in $\mathbb{P}$}\label{sec.dlsholding}

This section concentrates on those  distributive laws which hold in the complete Boolean algebra $\bB$, when $\kappa$ has countable cofinality.
Theorem \ref{thm.3.5} was proved by Prikry in the late 1960's; the first proof in print appears in this paper.
Here, we reproduce the main ideas of his proof, modifying his original argument slightly,
in particular,
using Lemma \ref{onlysomanypieces}, to simplify the  presentation.
In Theorem \ref{thm.3.9}  we prove  that $\bP$ satisfies a  Sacks-type property.  This, in turn, implies that the $(\om,\infty,<\kappa)$-d.l.\ holds in $\bB$ (Corollary \ref{cor.3.10}).
The reader is reminded that for the entire paper, Convention 
\ref{conventionforpaper} and Assumption \ref{basicassumption} are assumed.

\subsection{$(\om,\kappa_n)$-Distributivity}

\begin{definition}\label{defn.3.1}
A \defemph{stable tree system}
 is a pair $(F_N, F_\mbb{P})$ of functions
 $F_N : N \to N$ and
 $F_\mbb{P} : N \to \mbb{P}$,
 where $F_N$ is an embedding,
 such that
\begin{itemize}
\item[1)]
 For each $t \in N$,
 $\mathrm{Stem}(F_\mbb{P}(t)) \sqsupseteq F_N(t)$;
\item[2)]
 If $t_1 \in N$ is a proper initial segment of
 $t_2 \in N$, then
 $F_\mbb{P}(t_1) \supseteq F_\mbb{P}(t_2)$,
 and $F_N(t_1)$ is a proper initial segment of
 $F_N(t_2)$;
\item[3)]
 $F_N$ maps each level of $N$
 to a subset of a level of $N$
 (levels are mapped to distinct levels).
\end{itemize}
If requirement 3) is dropped,
 $(F_N, F_\mbb{P})$ is called a 
 \defemph{weak stable tree system}.
\end{definition}

Note that 1) can be rewritten as follows:
 $[F_\mbb{P}(t)] \subseteq B_{F_N(t)}$
 for all $t \in N$.
Note from 3) that
 $I(F``(N))$ is in $\mbb{P}$.

\begin{lemma}
\label{weakisenough}
Assume $\cf(\kappa) = \omega$.
If $(F_N, F_\mbb{P})$ is a weak stable tree system,
 then there is a tree $T \le N$ in strong splitting
 normal form and an embedding
 $F: N \to T$ such that
 $(F_N \circ F, F_\mbb{P} \circ F)$
 is a stable tree system.
\end{lemma}
\begin{proof}
Let $\Psi : N \to \omega$
 be the coloring
 $\Psi(u) := \dom(F_N(u))$.
Let $T \in \mbb{Q}$
 be given by Lemma~\ref{onlysomanypieces}.
Let $F : N \to T$ be an embedding
 that maps levels to levels.
The function $F$ is as desired.
\end{proof}

We point out that  Definition \ref{defn.3.1}  applies for $\kappa$ of any cofinality.
It can be shown that if
 $(F_N, F_\mbb{P})$ is a stable tree system
 and $\gamma < \cf(\kappa)$,
 then
 $$\bigcup \{ F_\mbb{P}(t) :
 t \in N(\gamma) \} \in \mbb{P}.$$
For our purposes, when $\cf(\kappa)=\om$, the following lemma will be useful.

\begin{lemma}
\label{intofunionisbig}
Assume $\cf(\kappa) = \omega$.
Let $(F_N, F_{\mbb{P}})$
 be a stable tree system.
Then $$T := \bigcap_{n < \omega}
 \bigcup \{ F_\mbb{P}(t) :
 t \in N(n) \}$$
 is in $\mbb{P}$.
Further, given any $S \le T$
 and $n \in \omega$,
 there is some $t \in N(n)$
 such that $S$ is compatible with
 $F_\mbb{P}(t)$.
\end{lemma}

\begin{proof}
To prove the first claim,
 note that
 $$T :=
 \bigcap_{n < \omega}
 \bigcup \{ F_\mbb{P}(t) : t \in N(n) \} =
 \bigcup_{f \in X}
 \bigcap_{n < \omega}
 F_\mbb{P}(f \restriction n).$$
This is because if
 $t_1, t_2 \in N$ are incomparable,
 then $F_\mbb{P}(t_1) \cap
 F_\mbb{P}(t_2) = \emptyset$.
Now temporarily fix $f \in X$.
One can see that
 $$\bigcap_{n < \omega}
 F_\mbb{P}(f \restriction n) =
 I(\{F_N( f \restriction n ) :
 n < \omega\}).$$
Now
 $$\bigcup_{f \in X}
 \bigcap_{n < \omega}
 F_\mbb{P}(f \restriction n) =
 \bigcup_{f \in X}
 I(\{F_N( f \restriction n ) :
 n < \omega\}) =
 I( F_N``(N)).$$
Thus, $T = I( F_N``(N))$,
 so $T$ is in $\mbb{P}$.

To prove the second claim,
 fix $S \le T$ and $n \in \omega$.
The stems of the trees
 $F_\mbb{P}(t)$ for $t \in N(n)$
 are pairwise incompatible.
Also, the stems of the trees
 $F_\mbb{P}(t)$ for $t \in N(n)$
 are all in $N(l)$ for some fixed
 $l \in \omega$.
Let $s \in S(l)$ be arbitrary.
Then $s = \mathrm{Stem}(F_\mbb{P}(t))$
 for some fixed $t \in N(n)$,
 and so $(S | s) \le F_\mbb{P}(t)$,
 showing that $S$ is compatible
 with $F_\mbb{P}(t)$.
\end{proof}

\begin{lemma}
\label{separate}
Assume $\cf(\kappa) = \omega$,
 and let $n < \omega$.
Consider any
 $\{ T_{\beta} \in \mbb{P} :
 \beta < \kappa_n \}$.
Then there is some
 $l < \omega$,
 a set $S \subseteq \kappa_n$
 of size $\kappa_n$,
 and an injection
 $J : S \to N(l)$
 such that
 $$(\forall \beta \in S)\,
 J(\beta) \in T_\beta.$$
\end{lemma}

\begin{proof}
For each $\beta < \kappa_n$,
 let $l_\beta < \omega$
 be such that $T_\beta$
 has $\ge \kappa_n$ nodes
 on level $l_\beta$.
Let $l < \omega$
 and $S \subseteq \kappa_n$
 be a 
 set of size $\kappa_n$
 such that $(\forall \beta \in S)\,
 l_\beta = l$;
these exist because $\kappa_n$
 is regular and $\omega < \kappa_n$.
Define the injection
 $J : S \to N(l)$
 by  mapping each element $\beta$ of
 $S$ to a node on level $l$
 of $T_\beta$ which is different from
 the nodes chosen so far.
Then $J$ satisfies the lemma.
\end{proof}


\begin{theorem}\label{thm.3.5}
Assume $\cf(\kappa) = \omega$.
Then  $\bP$ satisfies the $(\omega, \nu)$-\textnormal{d.l.}, for all $\nu<\kappa$.
\end{theorem}
\begin{proof}
Let $\mbb{B}$ be the complete
 Boolean algebra associated with $\mbb{P}$.
We have a dense embedding of $\mbb{P}$
 into $\mbb{B}$, which maps each condition
 $P \in \mbb{P}$ to the set
 of all conditions $Q \le P$.
Each element of $\mbb{B}$ is a
 downwards closed subset of $\mbb{P}$.
We shall show that for each $n<\om$, the $(\om,\kappa_n)$-d.l.\ holds in $\bB$.

Let $n<\om$ be fixed.
For each $m < \omega$, let
 $\langle a_{m,\gamma} \in \mbb{B} :
 \gamma < \kappa_n \rangle$
 be a maximal antichain in $\mbb{B}$.
For each $m < \omega$, the set
 $\bigcup \{
 a_{m,\gamma} : \gamma < \kappa_n \}$
 is dense in $\mbb{P}$.
To show that the specified
 distributive law holds, fix a
 non-zero element
 $b \in \mbb{B}$.
We must find a function
 $h \in {^{\omega} \kappa_n}$
 such that
 $$b \wedge \bigwedge_{m < \omega}
 a_{m,h(m)} > \mathbf{0}.$$
It suffices to show that for some
 $Q \in b$, there is a function
 $h \in {^\omega \kappa_n}$
 such that
 $$(\forall m < \omega)\,
 Q \in a_{m, h(m)}.$$

Fix any $P \in b$.
First, we will construct a stable tree
 system $(F_N, F_\mbb{P})$ with the property that
 $$(\forall m < \omega)
 (\forall t \in N(m))
 (\exists \gamma < \kappa_n)\,
 F_\mbb{P}(t) \in a_{m,\gamma}.$$
By Lemma~\ref{weakisenough},
 it suffices to define a weak stable tree system
 with this property.
To define $(F_N, F_\mbb{P})$,
 first let $F_N(\emptyset)$ be $\emptyset$ and
 $F_\mbb{P}(\emptyset) \le P$ be a member of
 $a_{0,\gamma}$ for some $\gamma < \kappa_n$.
Suppose that $t \in N$ and
 both $F_N(t)$ and $F_\mbb{P}(t)$ have been defined.
Suppose $t$ is on level $m$ of $N$.
Note that $\mathrm{Succ}_N(t) =
 \{ t ^\frown \beta : \beta < \kappa_m \}$.
For each $\beta < \kappa_m$,
 let $P_{\langle t, \beta \rangle}$
 be an element of
 $a_{m+1,\gamma}$ for some $\gamma < \kappa_n$.
We may apply Lemma~\ref{separate} to get
 injections $\eta_t : \mathrm{Succ}_N(t) \to \kappa_m$
 and $J_t : \mathrm{Succ}_N(t) \to N(l_t)$ for some
 $l_t < \omega$
 such that
 $(\forall s \in \mathrm{Succ}_N(t))\,
 J_t(s) \in P_{\langle t, \eta_t(s) \rangle}$.
For each $s \in \mathrm{Succ}(t)$,
 define $F_N(s) := J_t(s)$
 and $F_\mbb{P}(s) :=
 P_{\langle t, \eta_t(s) \rangle} | F_N(s)$.
Note that each $F_\mbb{P}(s)$ is in
 $a_{m+1,\gamma}$ for some
 $\gamma < \kappa_\alpha$.
Also, since the nodes
 $F_N(s) \sqsupseteq F_N(t)$
 for $s \in \mathrm{Succ}(t)$
 are pairwise incompatible,
 each $F_N(s)$ must be a
 \textit{proper} extension of $F_N(t)$.
This completes the definition of
 $(F_N, F_\mbb{P})$.

Let $\Psi : N \to \kappa_n$
 be the function such that for each
 $m < \omega$ and
 $t \in N(m)$,
 $\Psi(t) = \gamma < \kappa_n$
 is the unique ordinal such that
 $F_\mbb{P}(t) \in a_{m,\gamma}$.
Using the notation and result in
 Lemma~\ref{onlysomanypieces},
 there is some
 $h \in {^\omega \kappa_n}$
 such that $C_h$ includes
 a non-empty perfect set.
Fix such an $h$, and let
 $H \le N$ be a perfect tree
 such that $[H] \subseteq C_h$.
We have
 $$(\forall m < \omega)
  (\forall t \in H(m))\,
  F_\mbb{P}(t) \in a_{m,h(m)}.$$
Let $Q \in \mbb{P}$ be the set
 $$Q := \bigcap_{m < \omega}
 \bigcup \{ F_\mbb{P}(t) : t \in
 H(m) \}.$$
It is immediate that $Q \subseteq P$,
 because $F_\mbb{P}(\emptyset) = P$.
By Lemma~\ref{intofunionisbig},
 $Q \in \mbb{P}$.
Thus, $Q \le P$.

Now fix an arbitrary
 $m < \omega$.
We will show that
 $Q \in a_{m, h(m)}$,
 and this will complete the proof.
It suffices to show that for every
 $\gamma \not= h(m)$ and every
 $R \in a_{m, \gamma}$, we have
 $|[Q] \cap [R]| < \kappa^{\omega}$, as 
this will imply 
 there is no
 non-empty perfect
 subset of their intersection.

Fix such $\gamma$ and $R$.
We have $Q \le
 \bigcup \{ F_\mbb{P}(t) :
 t \in H(m) \}$.
In fact,
 $$[Q] \le
 \bigcup \{ [F_\mbb{P}(t)] :
 t \in H(m) \}.$$
Hence,
 $$[Q] \cap [R] \subseteq
 \bigcup \{ [F_\mbb{P}(t)] \cap [R] :
 t \in H(m) \}.$$
However,
 fix some $F_\mbb{P}(t)$
 for $t \in H(m)$.
The conditions $R \in a_{m,\gamma}$
 and $F_\mbb{P}(t) \in a_{m,h(m)}$
 are incompatible, so the closed set
 $[F_\mbb{R}(t)] \cap [R]$ must have size
 $\le \kappa$
 by Corollary~\ref{bigclosedhasperfect}.
We now have that
 $[Q] \cap [R]$ is a subset of
 a size $< \kappa$ union of
 size $\le \kappa$ sets.
Thus, $|[Q] \cap [R]| \le \kappa
 < \kappa^\omega$,
implying that the $(\om,\kappa_n)$-d.l.\ holds in $\bB$.
\end{proof}

\begin{question}
For $\cf(\kappa) > \omega$
 and $\nu<\kappa$,
 does $\bP$ satisfy the $(\cf(\kappa), \nu)$-d.l.?
\end{question}

\subsection{$(\om,\infty,<\kappa)$-Distributivity}

The next theorem we will prove
 will generalize the fact that
 $\mbb{P}$ satisfies the
 $(\omega, \kappa, {<\kappa})$-d.l.\
 (assuming $\cf(\kappa) = \omega)$.
The proof does not work for the
 $\cf(\kappa) > \omega$ case.
We could get the proof to work as long as
 we modified the forcing so that  fusion holds for  sequences of length
 $\cf(\kappa)$.
However, all such modifications we have tried
 cause important earlier theorems
 in this paper to fail.

\begin{definition}\label{defn.4.2}
Assume $\cf(\kappa) = \omega$.
A \defemph{fusion sequence} is a sequence
 of conditions $\langle T_n \in \mbb{P}
 : n < \omega \rangle$
 such that $T_0 \ge T_1 \ge ...$ and
 there exists a sequence of sets
 $\langle S_n \subseteq T_n : n < \omega \rangle$
 such that for each $n < \omega$,
 each $t \in S_n$ has
 $\ge \kappa_n$ successors in $T_n$,
 which are in $T_m$ for every $m \ge n$,
 and each successor of $t$ in $T_n$
 has an extension in $S_{n+1}$.
\end{definition}

\begin{lemma}\label{lem.4.3}
Let $\langle T_n \in \mbb{P} : n < \omega \rangle$
 be a fusion sequence and define
 $T_\omega := \bigcap_{n \in \omega} T_n$.
Then $T_\omega \in \mbb{P}$ and
 $(\forall n < \omega)\, T_\omega \le T_n$.
\end{lemma}
\begin{proof}
This is a standard argument.
\end{proof}

The following
 theorem shows that $\mbb{P}$
 has a property very similar to the Sacks property.

\begin{theorem}\label{thm.3.9}
Assume $\cf(\kappa) = \omega$.
Let $\mu : \omega \to ({\kappa - \{0\}})$ be any
 non-decreasing function
 such that $\lim_{n \to \omega} \mu(n) = \kappa$.
Let $\lambda = \kappa^\omega$.
Let $T \in \mbb{P}$ and $\dot{g}$ be such that
 $T \forces \dot{g} : \omega \to \check{\lambda}$.
Then there is some $Q \le T$
 and a function $f$ with domain $\omega$ such that
 for each $n \in \omega$,
 $|f(n)| \le \mu(n)$ and
 $Q \forces \dot{g}(\check{n}) \in \check{f}(\check{n})$.
\end{theorem}

\begin{proof}
We will define a decreasing (with respect to inclusion)
 sequence of trees
 $\langle T_n \in \mbb{P} : n \in \omega \rangle$
 such that some subsequence of this is a fusion sequence.
The condition $Q$ will be the intersection
 of the fusion sequence.
At the same time, we will define $f$.
For each $n \in \omega$ we will also define a set
 $S_n \subseteq T_n$ such that every child
 (in $T_n$) of every node in $S_n$
 will be in each tree $T_m$ for $m \ge n$.
Each node in $T_n$ will be comparable to some node in $S_n$.
Also, we will have $|S_n| \le \mu(n)$
 and each $t \in S_n$ will have
 $\le \mu(n)$ children in $T_n$.
Each element of $S_{n+1}$ will properly
 extend some element of $S_n$,
 and each element of $S_n$ will
 be properly extended by some element of $S_{n+1}$.

Let $S_0$ consist of a single node $t$ of $T$
 that has $\ge \kappa_0$ children.
Let $T' \subseteq T$ be a subtree such that
 $t$ is the stem of $T'$ and
 $t$ has exactly $\min\{\kappa_0, \mu(0)\}$ children.
For each $\gamma$ such that $t ^\frown \gamma \in T'$,
 let $U_{t ^\frown \gamma}$ be a subtree of
 $T | t ^\frown \gamma$ such that
 $U_{t ^\frown \gamma}$ decides the value of
 $\dot{g}(\check{0})$.
Let $T_0$ be the union of these $U_{t ^\frown \gamma}$ trees.
The condition $T_0$ allows for only $\le \mu(0)$
 possible values for $\dot{g}(\check{0})$.
Define $f(0)$ to be the set of these values.
We have $T_0 \forces \dot{g}(0) \in \check{f}(0)$.
Also, $|S_0| = 1$ and the unique node in $S_0$
 has $\le \mu(0)$ children in $T_0$,
 so $|f(0)| \le \mu(0)$.

Now fix $n > 0$ and suppose we have defined
 $T_0, ..., T_{n-1}$.
For each child $t \in T_{n-1}$
 of a node in $S_{n-1}$,
 pick an extension $s_t \in T_{n-1}$ of $t$
 that has $\ge \kappa_n$ children in $T_{n-1}$.
Let $S_n$ be the set of these $s_t$ nodes.
By hypothesis,
 $|S_{n-1}| \le \mu(n-1)$ and
 each node in $S_{n-1}$ has $\le \mu(n-1)$
 children in $T_{n-1}$,
Thus, $|S_n| \le \mu(n-1)$,
 and so $|S_n| \le \mu(n)$,
 because $\mu(n-1) \le \mu(n)$.
Let $T'_{n-1}$ be a subtree of $T_{n-1}$ such that
 each $s_t$ is in $T'_{n-1}$ and each $s_t$
 has exactly $\min \{ \kappa_n, \mu(n) \}$
 children in $T'_{n-1}$.
Thus, each $s_t \in S_n$ has
 $\le \mu(n)$ children in $T'_{n-1}$.
For each $s_t ^\frown \gamma$ in $T'_{n-1}$,
 let $U_{s_t ^\frown \gamma}$ be a subtree of
 $T'_{n-1} | s_t ^ \frown \gamma$
 that decides the value of
 $\dot{g}(\check{n})$.
Let $T_n$ be the union of the
 $U_{s_t ^\frown \gamma}$ trees.
We have $T_n \subseteq T'_{n-1} \subseteq T_{n-1}$.
The condition $T_n$ allows for only
 $\mu(n)$ possible values for $\dot{g}(\check{n})$.
Define $f(n)$ to be the set of these values.
We have that $|f(n)| \le \mu(n)$ and
 $T_n \forces \dot{g}(0) \in \check{f}(0)$.

This completes the construction of the
 sequence of trees and the function $f$.
Defining $Q := \bigcap_{n \in \omega} T_n$,
 we see that $Q$ is a condition because there
 is a subsequence of $\langle T_n : n \in \omega \rangle$
 that is a fusion sequence
 satisfying the hypothesis of the lemma above.
This is true because
 $\lim_{n \to \omega} \mu(n) = \kappa$.
The condition $Q$ forces the desired statements.
\end{proof}

Note that for the purpose of using the theorem above,
 each function $\mu' : \omega \to \kappa$
 such that $\lim_{n \to \omega} \mu'(n) = \kappa$
 everywhere dominates a non-decreasing function
 $\mu: \omega \to \kappa$ such that
 $\lim_{n \to \omega} \mu(n) = \kappa$.
Note also that
 nothing would have changed in the proof if instead we had
 $T \forces \dot{g} : \omega \to \check{V}$,
 because any name for an element of $V$
 can be represented by a function in $V$ from
 an antichain (which has size
 $\le \kappa^\omega$, by Proposition \ref{prop.2.31}) in $\mbb{P}$
 to $V$.

\begin{cor}\label{cor.3.10}
Assume $\cf(\kappa) = \omega$.
Then $\mbb{P}$ satisfies the
 $$(\omega, \infty, <\kappa)\mbox{-}\mathrm{d.l.}$$
\end{cor}


\section{Failures of  Distributive Laws}\label{sec.dlfails}

This section contains two of the three failures of distributive laws proved in this paper.
Here, we assume  Convention 
\ref{conventionforpaper} and Assumption \ref{basicassumption}, and do not place any restrictions on the cofinality of $\kappa$.
Theorems \ref{thm.4.1} and \ref{thm.4.10} were proved by Prikry in the late 1960's (previously unpublished)
for the case when $\cf(\kappa)=\om$, and here they are seen to easily generalize to $\kappa$ of any cofinality. 

\subsection{Failure of $(\cf(\kappa), \kappa, \kappa_n)$-Distributivity}

We point out that when
 $\cf(\kappa) = \omega$,
the $(\omega, \kappa, {<\kappa})$-d.l.\ holding in $\bP$ follows from the fact that $\bP$
 satisfies the
 $(\omega, \omega)$-d.l.
However, if we replace the third parameter
 ${< \kappa}$ with a fixed cardinal $\nu<\kappa$,
 the associated distributive law fails.
This is true in the $\cf(\kappa) > \omega$
 case as well.

\begin{theorem}\label{thm.4.1}
For each $\nu<\kappa$,
the $(\cf(\kappa), \kappa, \nu)$-\textnormal{d.l.}\
 fails for $\mbb{P}$.
\end{theorem}

\begin{proof}
It suffices to show that for each $\al<\cf(\kappa)$, the $(\cf(\kappa),\kappa,\kappa_{\al})$-d.l.\ fails in $\mathbb{P}$.
Note that a maximal antichain of $\mbb{P}$
 corresponds to a maximal antichain of the
 regular open completion of $\mbb{P}$, via mapping $P\in\bP$ to the regular open set $\{Q\in\bP:Q\le P\}$.
Let $\al<\cf(\kappa)$, and 
let $A_\beta := \{ (N | t) :
 t \in N(\beta) \}$
 for each $\beta < \cf(\kappa)$.
Each $A_\beta$ is a maximal antichain in $\mbb{P}$.
For each $\beta < \cf(\kappa)$,
 let $S_\beta \subseteq A_\beta$
 have size $\le \kappa_\alpha$.
Let $H \subseteq N$ be the set of $t$
 such that $N | t \in S_\beta$
 for some $\beta$.
Since each $S_\beta$ has size
 $\le \kappa_\alpha$,
 each level of $H$ has size
 $\le \kappa_\alpha$.
This implies that $H$ has at most
 $\kappa_\alpha^{\omega} < \kappa$ paths,
 and so $[H]$ cannot include a
 non-empty perfect subset.
By the definitions, we have
 $$H = \bigcap_{\beta < \cf(\kappa)}
 \bigcup S_\beta.$$
Since the left hand side of the equation
 above cannot include a perfect tree,
 neither can the right hand side.
Hence, the collection $A_{\beta}$, $\beta<\cf(\kappa)$, witnesses the failure of $(\cf(\kappa),\kappa,\kappa_{\al})$-distributivity in $\bP$.
\end{proof}


We point out that the previous theorem is stated in Theorem 4 (2) of \cite{Namba72}.
The  proof  there, though, is not obviously complete,  and for the sake of the literature and of full generality, the proof has been included here. 

\subsection{Failure of $(\mf{d}, \infty, <\kappa)$-Distributivity}

\begin{definition}
Given functions
 $f,g : \cf(\kappa) \to \cf(\kappa)$,
 we write $f \le^* g$
 and say $g$ \defemph{eventually dominates} $f$
 iff $$\{ \alpha < \cf(\kappa) :
 f(\alpha) > g(\alpha) \}$$
 is bounded below $\cf(\kappa)$.
Let $\mf{d}(\cf(\kappa))$ be the smallest size
 of a family of functions
 from $\cf(\kappa)$ to $\cf(\kappa)$
 such that each function from
 $\cf(\kappa)$ to $\cf(\kappa)$
 is eventually dominated by a member of this family.
\end{definition}

\begin{definition}
Let $\mc{D}$ be the collection of all
 functions $f$ from $\cf(\kappa)$ to $\cf(\kappa)$
 such that $f$ is
 non-decreasing and
 $$\lim_{\alpha \to \cf(\kappa)} f(\alpha) =
 \cf(\kappa).$$
We call a subset of $\mc{D}$ a
 \defemph{dominated-by} family
 iff given any function $g \in \mc{D}$,
 some function in the family is
 eventually dominated by $g$.
\end{definition}

The smallest size of a dominated-by family
 if $\mf{d}(\kappa)$.
We will prove the direction that for every
 dominating family, there is a dominated-by
 family of the same size.
The other direction is similar.
Let $\mc{F}$ be a dominating family.
Without loss of generality, each
 $f \in \mc{F}$ is strictly increasing.
Let $\mc{F}' := \{ f' : f \in \mc{F} \}$,
 where each $f'$ is a non-decreasing
 function that extends the partial function
 $\{ (y,x) : (x,y) \in f \}$.
Since $\mc{F}$ is a dominating family,
 it can be shown that $\mc{F}'$
 is a dominated-by family.

\begin{definition}
Given $f \in \mc{D}$,
 we say that a perfect tree
 $T \in \mbb{P}$
 \defemph{obeys} $f$ iff
 for each $\alpha < \cf(\kappa)$,
 the $\alpha$-th level of $T$
 has $\le \kappa_{f(\alpha)}$
 nodes in $T$.
\end{definition}

\begin{lemma}\label{lem.4.5}
Let $\lambda = \mf{d}(\cf(\kappa))$
 and $G =
 \{ g_\gamma \in \mc{D} : \gamma < \lambda \}$
 be a dominated-by family.
Then there is some $\delta < \cf(\kappa)$
 such that
 $$(\forall \alpha < \cf(\kappa))
   (\exists \gamma \in \lambda)\,
   g_\gamma(\alpha) \le \delta.$$
\end{lemma}

\begin{proof}
Assume there is no such
 $\delta < \cf(\kappa)$.
For each $\delta < \cf(\kappa)$,
 let $\alpha_\delta < \cf(\kappa)$
 be the least ordinal such that
$$
(\forall \gamma<\lambda)\ g_{\gamma}(\al_{\delta})>\delta.
$$
It must be that
 $\delta_1 < \delta_2$ implies
 $\alpha_{\delta_1} \le
  \alpha_{\delta_2}$.
Now, the limit
 $$\mu := \lim_{\delta \to \cf(\kappa)}
 \alpha_\delta$$
 cannot be less than $\cf(\kappa)$.
To see why, suppose $\mu < \cf(\kappa)$.
Consider $g_0$.
The function $g_0 \restriction (\mu+1)$
 must be bounded below $\cf(\kappa)$,
since $\cf(\kappa)$ is regular.
Let $\delta$ be such a bound.
Since $\alpha_\delta \le \mu$
 and $g$ is non-decreasing,
 we have
 $g_0(\alpha_\delta) \le g(\mu) \le \delta$,
 which contradicts the definition
 of $\alpha_\delta$.

We have now shown that
 $\mu = \cf(\kappa)$.
The partial function
 $\alpha_\delta \mapsto \delta$
 may not be well-defined.
To fix this problem, for each
 $\alpha$ which equals
 $\alpha_\delta$ for at least
 one value of $\delta$,
 pick the least such $\delta$.
Let $\Delta \subseteq \cf(\kappa)$
 be the cofinal set of
 such $\delta$ values picked.
This results in a well-defined
 partial function
 which is non-decreasing.
Let $f \in \mc{D}$ be an extension
 of this partial function.
Since $G$ is a dominated-by family,
 fix some $\gamma$ such that
 $f$ dominates $g_\gamma$.
Now, let $\delta \in \Delta$
 be such that
 $g_\gamma(\alpha_\delta)
 \le f(\alpha_\delta)$.
Since $f(\alpha_\delta) = \delta$,
 we get that
 $g_\gamma(\alpha_\delta) \le \delta$,
 which contradicts the definition of
 $\alpha_\delta$.
\end{proof}


\begin{theorem}\label{thm.4.10}
The $(\mf{d}(\cf(\kappa)),
 \infty,
 <\kappa)$\textnormal{-d.l.} fails for $\mbb{P}$.
\end{theorem}

\begin{proof}
Let $\lambda = \mf{d}(\cf(\kappa))$.
Let $\{ f_\gamma \in \mc{D} : \gamma < \lambda \}$
 be a set which forms a
 dominated-by family.
For each $\gamma < \lambda$, let
 $\mc{A}_\gamma \subseteq \mbb{P}$
 be a maximal antichain in $\mbb{P}$
 with the property that for each
 $T \in \mc{A}_\gamma$,
 $T$ obeys $f_\gamma$.
Note that each $\mc{A}_\gamma$
 has size $\le \kappa^{\cf(\kappa)}
 = |\mbb{P}|$.

For each $\gamma < \lambda$,
 let $\mc{B}_\gamma \subseteq \mc{A}_\gamma$
 be some set of size strictly less than $\kappa$.
Let $u : \mbb{P} \to \mbb{B}$
 be the standard embedding of $\mbb{P}$
 into its completion.
We claim that
 $$\bigwedge_{\gamma < \lambda}
   \bigvee
  \{ u(T) : T \in \mc{B}_\gamma \} = 0,$$
 which will prove the theorem.
To prove this claim,
 for each $\gamma < \lambda$ let
 $$T_\gamma := \bigcup \mc{B}_\gamma.$$
The claim will be proved once we
 show that
 $\tilde{T} := \bigcap_{\gamma < \lambda}
 T_\gamma$
 does not include a perfect tree.
It suffices to find some
 $\delta < \cf(\kappa)$
 such that there is a cofinal set of
 levels of $\tilde{T}$ that each
 have $\le \kappa_\delta$ nodes.

Since $\cf(\kappa) < \lambda$
 are both regular cardinals,
 fix a set
 $K \subseteq \cf(\kappa)$
 of size $\cf(\kappa)$
 and some $\delta < \cf(\kappa)$
 such that
 $|\mc{B}_\gamma| \le \kappa_\delta$
 for each $\gamma \in K$.
Given $\gamma \in K$, define
 $g_\gamma \in D$ to be the function
 $$g_\gamma(\alpha) := \max \{
 f_\gamma(\alpha), \delta \}.$$
As
 $|\mc{B}_\gamma| \le
 \kappa_\delta$ and
 $(\forall T \in \mc{B}_\gamma)\,$
 $T$ obeys $f_\gamma$,
 it follows that
 $T_\gamma = \bigcup \mc{B}_\gamma$
 obeys $g_\gamma$.
Thus, by the definition of
 $\tilde{T}$, it suffices to find
 a cofinal set $L \subseteq \cf(\kappa)$
 and for each $l \in L$ an ordinal
 $\gamma_l \in K$ such that
 $g_{\gamma_l}(l) \le \delta$.
This, however, follows
 from Lemma \ref{lem.4.5}.
\end{proof}

For $\cf(\kappa)=\om$, assuming the Continuum Hypothesis and that $2^{\kappa}=\kappa^+$, 
 Theorem 4 (4) of \cite{Namba72}
states that for all $\om\le \lambda\le\kappa^+$, the $(\om_1,\lambda,<\lambda)$-d.l.\ fails in $\bP$.
Under these assumptions, that theorem of Namba implies Theorem \ref{thm.4.10}.
We have included our proof as it is simpler and the result is more general than that in \cite{Namba72}.


\section{$\mathcal{P}(\omega)/\Fin$ and $\mf{h}$}\label{sec.5}

In this section, we show that the Boolean algebra $\mathcal{P}(\omega)/\Fin$ completely embeds into $\bB$.
Similar reasoning shows that the forcing 
$\bP$ collapses the cardinal $\kappa^{\om}$ to the distributivity number $\mathfrak{h}$.
It will follow that the $(\mathfrak{h},2)$-distributive law fails in $\bB$;
hence assuming the Continuum Hypothesis, $\bB$ does not satisfy the $(\om_1,2)$-d.l.
Similar results were proved by  \Bukovsky\ and \Coplakova\  in Section 5 of \cite{Bukovsky/Coplakova90}.
They considered  perfect trees, where there is a 
 fixed family of countably many regular cardinals and for each cardinal $\kappa_n$ in the family, their perfect trees must have cofinally many levels where the branching has size $\kappa_n$; similarly for their family of Namba forcings.

Recall that the regular open completion of
 a poset is the collection of regular open subsets
 of the poset ordered by inclusion.
For simplicity, we will work with the
 poset $\mbb{P}'$ of conditions in $\mbb{P}$
 that are in strong splitting formal form.
$\mbb{P}'$ forms a dense subset of $\mbb{P}$,
 so $\mbb{P}'$ and $\mbb{P}$ have isomorphic
 regular open completions.
For this section,
 let $\mbb{B}'$ denote the regular open completion
 of $\mbb{P}'$
 (and $\mbb{B}$ is the regular open completion of $\mbb{P}$).
Recall the following definition:
\begin{definition}
\label{cedef}
Let $\mbb{S}$ and $\mbb{T}$ be complete Boolean algebras.
A function $i : \mbb{S} \to \mbb{T}$ is a
 \defemph{complete embedding}
 iff the following are satisfied:
\begin{itemize}
\item[1)] $(\forall s, s' \in \mbb{S}^+)\,
 s' \le s \Rightarrow i(s') \le i(s)$;
\item[2)] $(\forall s_1, s_2 \in \mbb{S}^+)\,
 s_1 \perp s_2 \Leftrightarrow i(s_1) \perp i(s_2)$;
\item[3)] $(\forall t \in \mbb{T}^+)
 (\exists s \in \mbb{S}^+)
 (\forall s' \in \mbb{S}^+)\,
 s' \le s \Rightarrow
 i(s') || t$.
\end{itemize}
\end{definition}

If $i : \mbb{S} \to \mbb{T}$
 is a complete embedding,
 then if $G$ is $\mbb{T}$-generic over $V$,
 then there is some $H \in V[G]$ that is
 $\mbb{S}$-generic over $V$.

\begin{definition}
Given $T \in \mbb{P}$,
 $\mathrm{Split}(T) \subseteq \omega$
 is the set of $l \in \omega$
 such that $T$ has
 a splitting node
 on level $l$.
\end{definition}

\begin{theorem}\label{thm.5.3}
There is a complete embedding of
 $\mathcal{P}(\omega)/\Fin$ into $\mbb{B}$.
\end{theorem}
\begin{proof}
It suffices to show there is a complete embedding
 of $P(\omega)/\Fin$ into $\mbb{B}'$.
For each $X \in [\omega]^\omega$, define
 $\mc{S}_X \in \mbb{B}'$ to be
 $\mc{S}_X := \{ T \in \mbb{P}' :
 \mathrm{Split}(T) \subseteq^* X \}$
Note that $X =^* X'$ implies
 $\mc{S}_X = \mc{S}_{X'}$.
Define $i : [\omega]^\omega \to \mbb{P}'$
 to be $i(X) := \mc{S}_X$.
This induces a map from
 $P(\omega)/\Fin$ to $\mbb{B}'$.
We will show this is a complete embedding.

First, we must establish that each
 $\mc{S}_X$ is indeed in $\mbb{B}'$.
Temporarily fix $X \in [\omega]^\omega$.
We must show that
 $\mc{S}_X \subseteq \mbb{P}'$
 is a regular open subset of $\mbb{P}'$.
First, it is clear that $\mc{S}_X$
 is closed downwards.
Second, consider any
 $T_1 \not\in \mc{S}_X$.
By definition,
 $|\mathrm{Split}(T_1) - X| = \omega$.
By the nature of strong
 splitting normal form,
 there is some $T_2 \le T_1$ in $\mbb{P}'$
 such that $\mathrm{Split}(T_2) = \mathrm{Split}(T_1) - X$.
We see that for each
 $T_3 \le T_2$ in $\mbb{P}'$,
 $T_3 \not\in \mc{S}_X$.
Thus, $\mc{S}_X$ is a regular open set.


We will now show that $i$
 induces a complete embedding.
To show 1) of Definition~\ref{cedef},
 suppose $Y \subseteq^* X$ are in $[\omega]^\omega$.
If $T \in \mc{S}_Y$,
 then $\mathrm{Split}(T) \subseteq^* Y$,
 so $\mathrm{Split}(T) \subseteq^* X$,
 which means $T \in \mc{S}_X$.
Thus, $\mc{S}_Y \subseteq \mc{S}_X$,
 so 1) is established.

To show 2) of the definition,
 suppose $X, Y \in [\omega]^\omega$ but
 $X \cap Y$ is finite.
Suppose, towards a contradiction,
 that there is some
 $T \in \mc{S}_X \cap \mc{S}_Y$.
Then $\mathrm{Split}(T) \subseteq^* X$
 and $\mathrm{Split}(T) \subseteq^* Y$, so 
 $\mathrm{Split}(T) \subseteq^* X \cap Y$,
 which is impossible because
 $\mathrm{Split}(T)$ is infinite.

To show 3) of the definition,
 fix $T_1 \in \mbb{P}$.
Let $X := \mathrm{Split}(T_1)$.
We will show that for each
 infinite $Y \subseteq^* X$,
 there is an extension of $T_1$
 in $\mc{S}_Y$.
Fix an infinite $Y \subseteq^* X$.
By the nature of strong splitting normal
 form, there is some $T_2 \le T_1$ such that
 $\mathrm{Split}(T_2) = Y \cap X$.
Thus, $T_2 \in \mc{S}_Y$.
This completes the proof.
\end{proof}

\begin{cor}
Forcing with $\mbb{P}$ adds
 a selective ultrafilter on $\omega$.
\end{cor}
\begin{proof}
Forcing with $\mathcal{P}(\omega)/\Fin$
 adds a selective ultrafilter.
\end{proof}

\begin{definition}
The distributivity number,
 denoted $\mf{h}$,
 is the smallest ordinal $\lambda$
 such that the
 $(\lambda,\infty)$-d.l.\
 fails for $\mathcal{P}(\omega)/\Fin$.
\end{definition}

We have that
 $\omega_1 \le \mf{h} \le 2^\omega$.
The $(\mf{h},2)$-d.l.\ in fact
 fails for $\mathcal{P}(\omega)/\Fin$.
Thus, forcing with $\mbb{P}$ adds a new
 subset of $\mf{h}$.
It is also well-known
 (see \cite{BlassHB})
 that forcing with $\mathcal{P}(\omega)/\Fin$
 adds a surjection from $\mf{h}$ to $2^\omega$.
Thus, forcing with $\mbb{P}$ collapses
 $2^\omega$ to $\mf{h}$.
We will now see that many more cardinals
 get collapsed to $\mf{h}$.

\begin{definition}
A family $\mc{H} \subseteq [\omega]^\omega$
 is called \defemph{almost disjoint}
 iff the intersection of any two elements of $\mc{H}$ is finite.
A family $\mc{H}\subseteq [\omega]^\omega$
 is called \defemph{mad}
 (maximally almost disjoint)
 iff $\mc{H}$ is almost disjoint and there is no
 almost disjoint family $\mc{H}'$ such that
 $\mc{H} \subsetneq \mc{H}' \subseteq [\omega]^\omega$.
\end{definition}

\begin{definition}
A \defemph{base matrix tree}
 is a collection $\{ \mc{H}_\alpha : \alpha < \mf{h} \}$
 of mad families $\mc{H}_\alpha \subseteq [\omega]^\omega$
 such that $\bigcup_{\alpha < \mf{h}} \mc{H}_\alpha$
 is dense in $[\omega]^\omega$
 with respect to almost inclusion.
\end{definition}

Balcar, Pelant and Simon proved in \cite{Balcar/Pelant/Simon80} that a base matrix for $\mathcal{P}(\omega)/\Fin$ exists, assuming only ZFC.
The following lemma and theorem use ideas from the proof of Theorem 5.1 in \cite{Bukovsky/Coplakova90}, in which \Bukovsky\ and \Coplakova\ prove that their perfect tree forcings, described above, collapses $\kappa^+$ to $\mathfrak{h}$, assuming $2^{\kappa}=\kappa^+$.

\begin{lemma}
There exists a family
 $\{ \mc{A}_\alpha \subseteq \mbb{P} : \alpha < \mf{h} \}$
 of maximal antichains such that
 $\bigcup_{\alpha < \mf{h}} \mc{A}_\alpha$
 is dense in $\mbb{P}$.
\end{lemma}
\begin{proof}
Let $\{ \mc{H}_\alpha \subseteq [\omega]^\omega :
 \alpha < \mf{h} \}$
 be a base matrix tree.
For an infinite $A \subseteq \omega$,
 let $\mbb{P}_A :=
 \{ T \in \mbb{P} : \mathrm{Split}(T) \subseteq A \}$.
For an infinite $A \subseteq \omega$,
 we may easily construct an antichain
 $\mc{B}_A \subseteq \mbb{P}_A$
 whose downward closure is dense in
 $\mbb{P}_A$.
Now temporarily fix $\alpha < \mf{h}$.
For distinct $A_1, A_2 \in \mc{H}_\alpha$,
 the elements of $\mc{B}_{A_1}$ are incompatible
 with the elements of $\mc{B}_{A_2}$,
 because if $T_1 \in \mc{B}_{A_1}$ and
 $T_2 \in \mc{B}_{A_2}$, then
 $\mathrm{Split}(T_1) \subseteq^* A_1$ and
 $\mathrm{Split}(T_2) \subseteq^* A_2$,
 so $T_1$ and $T_2$ cannot have a common
 extension because $A_1 \cap A_2$ is finite.

For each $\alpha < \mf{h}$,
 define $\mc{A}_\alpha :=
 \bigcup \{ \mc{B}_A : A \in \mc{H}_\alpha \}$.
Temporarily fix $\alpha < \mf{h}$.
We will show that $\mc{A}_\alpha$ is maximal.
Consider any $T \in \mbb{P}$.
We will show that some extension of
 $T$ is compatible with an element of $\mc{A}_\alpha$.
Let $T' \le T$ be such that
 $\mathrm{Split}(T') \subseteq A$
 for some fixed $A \in \mc{H}_a$.
If there was no such $A$, then
 $\mathrm{Split}(T)$ would witness that $\mc{H}_\alpha$
 is not a mad family.
Hence, $T' \in \mbb{P}_A$.
Since the downward closure of $\mc{B}_A$
 is dense in $\mbb{P}_A$,
 we have that $T'$ (and hence $T$) is compatible with some
 element of $\mc{B}_A \subseteq \mc{A}_\alpha$.

We will now show that
 $\bigcup_{\alpha < \mf{h}} \mc{A}_\alpha$
 is dense in $\mbb{P}$.
Fix any $T \in \mbb{P}$.
Let $A \in \bigcup_{\alpha < \mf{h}} \mc{H}_\alpha$
 be such that $A \subseteq^* \mathrm{Split}(T)$.
Let $T' \le T$ be such that $\mathrm{Split}(T') \subseteq
 A \cap \mathrm{Split}(T)$,
 and let $S \in \mc{B}_A$ be such that $S \le T'$.
Then $S \le T$,
 and we are finished.
\end{proof}

\begin{theorem}
The forcing $\mbb{P}$ collapses
 $\kappa^\omega$ to $\mf{h}$.
\end{theorem}
\begin{proof}
We work in the generic extension.
Let $G$ be the generic filter.
By the previous lemma,
 let $\{ \mc{A}_\alpha \subseteq \mbb{P}
 : \alpha < \mf{h} \}$
 be a collection of maximal antichains such that
 $\bigcup_{\alpha < \mf{h}} \mc{A}_\alpha$
 is dense in $\mbb{P}$.
For each $T \in
 \bigcup_{\alpha < \mf{h}} \mc{A}_\alpha$,
 let $F_T : \kappa^\omega \to \mbb{P}$
 be an injection such that
 $\{ F_T(\beta) : \beta < \kappa^\omega \}$
 is a maximal antichain below $T$
 (which exists by Lemma~\ref{splititup}).
Consider the function $f :
 \mf{h} \to \kappa^\omega$
 defined by
 $$f(\alpha) := \beta \Leftrightarrow
 (\exists T \in \mbb{P})\,
 T \in \mc{A}_\alpha \cap G \mbox{ and }
 F_T(\beta) \in G.$$
This is indeed a function because
 for each $\alpha$,
 there is at most one $T$ in
 $\mc{A}_\alpha \cap \mathrm{G}$,
 and there is at most one
 $\beta < \kappa^\omega$ such that
 $F_T(\beta) \in G$.

To show that $f$ surjects onto $\kappa^\omega$,
 fix $\beta < \kappa^\omega$.
We will find an $\alpha < \mf{h}$ such that
 $f(\alpha) = \beta$.
It suffices to show that
 $$\{ F_T(\beta) : T \in \bigcup_
 {\alpha < \mf{h}} \mc{A}_\alpha \}$$
 is dense in $\mbb{P}$.
To show this, fix $S \in \mbb{P}$.
Since $\bigcup_{\alpha < \mf{h}} \mc{A}_\alpha$
 is dense in $\mbb{P}$,
 fix some $\alpha < \mf{h}$ and $T \in \mc{A}_\alpha$
 such that $T \le S$.
We have $F_T(\beta) \le T$,
 so $F_T(\beta) \le S$
 and we are done.
\end{proof}


\section{Minimality of $\omega$-Sequences}\label{sec.6}

For the entire section,
 we will assume $\cf(\kappa) = \omega$.
Sacks forcing was the first forcing shown to add a minimal degree of constructibility.
In \cite{Sacks}, Sacks proved that given a generic filter $G$ for the perfect tree forcing on
 ${^{<\omega} 2}$,
each real $r:\om\ra 2$ in $V[G]$ which is not in $V$ can be used to reconstruct the generic filter $G$.
A forcing {\em adds a minimal degree of constructibility} if 
whenever  $\dot{A}$ is a name
 forced by a condition $p$ to be
 a function from an ordinal to $2$,
 then $p \forces (\dot{A} \in \check{V} \mbox{ or }
 \dot{G} \in \check{V}(\dot{A}))$,
 where $\dot{G}$ is the name for the generic filter
 and $1 \forces \check{V}(\dot{A})$
 is the smallest inner model $M$ such that
 $\check{V} \subseteq M$ and $\dot{A} \in M$.

One may also ask whether the generic extension is minimal with respect to adding new sequences from $\omega$ to a given cardinal.
Abraham \cite{Abraham85} and Prikry proved that the perfect tree forcings and the version of Namba forcing involving subtrees of
 ${^{< \omega} \omega_1}$
 thus adding an unbounded function from $\om$ into $\om_1$ are minimal, assuming $V=L$ (see Section 6 of \cite{Bukovsky/Coplakova90}). 
Carlson, Kunen and Miller showed this to be the case assuming Martin's Axiom and the negation of the Continuum Hypothesis in \cite{Carlson/Kunen/Miller84}.
The question of minimality  was investigated generally for two models of ZFC $M\subseteq N$ (not necessarily forcing extensions) when $N$ contains a new subset of a cardinal regular in $M$ in Section 1 of \cite{Bukovsky/Coplakova90}.
In Section 6 of that paper, 
\Bukovsky\ and \Coplakova\  proved  that  their families of perfect tree and generalized Namba forcings are  minimal with respect to adding new $\om$-sequences of ordinals, but do not produce minimal generic extensions, since $\mathcal{P}(\om)/\Fin$ completely embeds into their forcings.

Brown and Groszek investigated the question of minimality of forcing extensions
 was investigated for forcing posets consisting of superperfect subtrees of
 ${^{<\kappa}\kappa}$,
 where 1) $\kappa$ is an uncountable regular cardinal,
 2) splitting along any branch forms a club set of levels,
 3) and whenever a node splits, its immediate successors are in some $\kappa$-complete, nonprincipal normal filter.
In \cite{Brown/Groszek06}, they proved that this forcing adds a generic of minimal degree if and only if the filter is $\kappa$-saturated.

In this section, 
we show that, assuming that $\kappa$ is a limit of measurable cardinals,
 $\mbb{P}$
 is minimal with respect to
 $\omega$-sequences, meaning if
 $p \forces \dot{A} : \omega \to \check{V}$, then
 $(p \forces \dot{A} \in \check{V}$ or
 $\dot{G} \in \check{V}(\dot{A}))$.
$\bP$ 
does not add a minimal degree of constructibility,
since $\mathcal{P}(\omega)/\Fin$ completely embeds into $\bB$, and that intermediate model has no new $\om$-sequences.

The proof that Sacks forcing $\mbb{S}$ is minimal
 follows once we observe that given an ordinal
 $\alpha$, a name $\dot{A}$ such that
 $p \forces \dot{A} \in {^{\check{\alpha}} 2} - \check{V}$,
 and two conditions $p_1, p_2$,
 there are $p_1' \le p_1$ and $p_2' \le p_2$
 that decide $\dot{A}$ to extend incompatible
 sequences in $V$.
After this observation,
 given any condition $p \in \mbb{S}$, we can
 extend $p$ using fusion to get $q \le p$
 so that which branch
 the generic is through $q$ can be recovered by knowing
 which initial segments (in $V$) the sequence $\dot{A}$
 extends.
This is because every child of a splitting node in $q$
 has been tagged with a sequence in $V$,
 and no two children of a splitting node are tagged
 with compatible sequences.

In Sacks forcing $\mbb{S}$,
 every node has at most $2$ children.
In our forcing $\mbb{P}$ (assuming $\cf(\kappa) = \omega$),
 for each $n < \omega$
 there must be some nodes that have $\ge \kappa_n$ children. 
To make the proof work for $\mbb{P}$, we would like
 that whenever $n < \omega$ and
 $\langle p_\gamma \in \mbb{P} : \gamma < \kappa_n \rangle$
 is a sequence
 of conditions each forcing $\dot{A}$ to be in
 ${^{\check{\alpha}}2} - \check{V}$,
 then there exists a set of pairwise incompatible sequences
 $\{ s_\gamma \in {^{<{\alpha}} 2} : \gamma < \kappa_n \}$ and
 a set of conditions $\{ p_\gamma' \le p_\gamma : \gamma < \kappa_n \}$
 such that
 $(\forall \gamma < \kappa_n)\,
  p_\gamma' \forces
 \check{s}_\gamma
 \sqsubseteq
 \dot{A}$.
However,
 suppose $1 \forces \dot{A} \in {^{\check{\omega}_1} 2}$,
 $2^{< \omega_1} = 2^\omega < \kappa_0$,
 and $\kappa_0$ is a measurable cardinal
 as witnessed by some normal measure.
Then there is a measure one set of
 $\gamma \in \kappa_0$ such that the $s_\gamma$ are all the same.

Thus,
 when we shrink a tree to try to
 assign tags to its nodes,
 there seems to be the possibility that we can shrink it further to cause the resulting tags to give us no information.
There is a special case:
 if $1 \forces \dot{A} : \omega \to \check{V}$
 and $1 \forces \dot{A} \not\in \check{V}$,
 then it is impossible to perform fusion to
 decide more and more of $\dot{A}$ while at the same time
 shrinking to get tags that are identical for each
 stage of the fusion.
The intersection of the fusion sequence would be
 a condition $Q$ such that
 $Q \forces \dot{A} \in \check{V}$,
 which would be a contradiction.
The actual proof by contradiction
 uses a thinning procedure more complicated than
 ordinary fusion.
Our proof will make the special assumption that
 $\kappa$ is a limit of measurable cardinals
 to perform the thinning.

When we say ``thin the tree $T$'',
 it is understood that we mean get
 a subtree $T'$ of $T$ that is still perfect,
 and replace $T$ with $T'$.
When we say ``thin the tree $T$ below $t \in T$'',
 we mean thin $T|t$ to get some $T'$,
 and then replace $T$
 by $T' \cup \{ s \in T : s$ is incompatible with $t \}$.

\begin{definition}
Fix a name $\dot{A}$
 such that $1_\mbb{P} \forces
 \dot{A} : \omega \to \check{V}$ and
 $1_\mbb{P} \forces \dot{A} \not\in \check{V}$.
For each condition $T \in \mbb{P}$,
 let $\psi_T : T \to {^{<\omega} V}$ be the function
 which assigns to each node $t \in T$
 the longest sequence $s = \psi_T(t)$ such that
 $(T | t) \forces \dot{A} \sqsupseteq \check{s}$.
Call a splitting node
 $t \in T$ a \defemph{red} node of $T$ iff the sequences
 $\psi_T(c)$ for $c \in \mathrm{Succ}_T(t)$ are all the same.
Call a splitting node
 $t \in T$ a \defemph{blue} node of $T$ iff the sequences
 $\psi_T(c)$ for $c \in \mathrm{Succ}_T(t)$ are pairwise incomparable,
 where we say two sequences are incomparable iff neither
 is an end extension of the other.
\end{definition}

Although $\psi_T$ and the notions of a red and blue node
 depend on the name $\dot{A}$,
 in practice there will be no confusion.
Note that being blue is preserved when we pass to
 a stronger condition but being red may not be.
For the sake of analyzing the minimality of $\mbb{P}$
 with respect to $\omega$-sequences,
 we want to be able to shrink any perfect tree $T$
 to get some perfect $T' \le T$
 whose splitting nodes are all blue:

\begin{lemma}[Blue Coding]
\label{bluecoding}
Let $T \in \mbb{P}$, $\dot{A}$, and $\alpha \in \mathrm{Ord}$
 be such that $T \forces (\dot{A} : \check{\alpha} \to \check{V})$
 and $T \forces \dot{A} \not\in \check{V}$.
Suppose the following are satisfied:
\begin{itemize}
\item[1)]
T is in weak splitting normal form.
\item[2)]
Each splitting node of $T$ is a blue node of $T$.
\end{itemize}
Then $T \forces \dot{G} \in \check{V}(\dot{A})$,
 where $\dot{G}$ is the generic filter.
\end{lemma}
\begin{proof}
Unlike almost every other proof in this paper,
 we will work in the extension.
Let $G$ be the generic filter,
 $g := \bigcap G$,
 $\check{V}_G$
 be the ground model, and $\dot{A}_G$
 be the interpretation of the name $\dot{A}$.
It suffices to prove how $g$ can be constructed
 from $\dot{A}_G$ and $\check{V}_G$.
We have that $g$ is a path through $T$.
Let $t_0$ be the stem of $T$.
Now $g$ must extend one of the children of $t_0$ in $T$.
Because $t_0$ is blue in $T$,
 this child $c$ can be defined
 as the unique $c \in \mathrm{Succ}_T(t_0)$ satisfying
 $\psi_T(c) \sqsubseteq \dot{A}_G$.
Call this child $c_0$.
Now let $t_1$ be the unique minimal extension of $c_0$
 that is splitting.
In the same way, we can define the
 $c \in \mathrm{Succ}_T(t_1)$ that $g$ extends as
 the unique child $c$ that satisfies
 $\psi_T(c) \sqsubseteq \dot{A}_G$.
Call this child $c_1$.
We can continue like this, and the sequence
 $c_0 \sqsubseteq c_1 \sqsubseteq c_2 \sqsubseteq ...$
 is constructible
 from $\check{V}_G$ and $\dot{A}_G$.
Since $g$ is the unique path that extends each $c_i$,
 we have that $g$ is constructible
 from $\check{V}_G$ and $\dot{A}_G$
 (and so $G$ is as well).
\end{proof}

\begin{lemma}[Blue Selection]
\label{blueselection}
Let $\lambda_1 < \lambda_2$ be cardinals.
Suppose there is an ultrafilter $\mc{U}$ on $\lambda_2$
 that is uniform and $\lambda_1$-complete
 (which happens if $\lambda_2$ is a measurable cardinal
 and $\mc{U}$ is a normal ultrafilter on $\lambda_2$).
Let $\langle S_\alpha \in [
 \bigcup_{\gamma \in \mathrm{Ord}} {^\gamma V}]^{\lambda_2} :
 \alpha < \lambda_1 \rangle$ be a $\lambda_1$-sequence
 of size $\lambda_2$ sets of sequences,
 where within each $S_\alpha$ the sequences are pairwise
 incomparable.
Then there is a sequence
 $\langle a_\alpha \in S_\alpha : \alpha < \lambda_1 \rangle$
 such that the $a_\alpha$ are pairwise incomparable.
\end{lemma}
\begin{proof}
The ultrafilter $\mc{U}$ on $\lambda_2$
 induces an ultrafilter on each $S_\alpha$,
 so we may freely talk about a measure one subset of $S_\alpha$.
Given sequences $a, b$, we write $a || b$ to mean they
 are comparable (one is an initial segment of the other).

\underline{Claim 1}:
Fix $\alpha_1, \alpha_2 < \lambda_1$.
Then there is at most one $a \in S_{\alpha_1}$ such that
 $B_a := \{ b \in S_{\alpha_2} : a || b \}$ has measure one.

\underline{Subclaim}:
Suppose $a \in S_{\alpha_1}$ is such that $B_a$ has measure one.
Then all elements of $B_a$ extend $a$.
To see why, suppose there is some $b \in B_a$
 which does not extend $a$.
Then $b$ is an initial segment of $a$.
Let $b'$ be another element of $B_a$.
Since $b \perp b'$,
 it must be that $a \perp b'$, which is a contradiction.

Towards proving Claim 1, suppose $a, a'$ are distinct elements
 of $S_{\alpha_1}$ such that the sets $B_a$ and $B_{a'}$ have
 measure one.
There must be some $b \in B_a \cap B_{a'}$.
We have that $b$ extends both $a$ and $a'$,
 which is impossible because $a \perp a'$.
This proves Claim 1.

We will now prove the theorem.
For each $\alpha_1, \alpha_2 < \lambda_1$,
 remove the unique element of $S_{\alpha_1}$
 that is comparable with measure one elements of
 $S_{\alpha_2}$
 (if it exists).
This replaces each set $S_\alpha$ with
 a new set $S_\alpha'$.
Since $\lambda_1 < \lambda_2$
 and the ultrafiler $\mc{U}$ on $\lambda_2$ is uniform, each
 $S_\alpha'$ has size $\lambda_2$
 (and is concentrated on by the
 ultrafilter on $S_\alpha$).
Let $a_0$ be any element of $S_0'$.
Now fix $0 < \alpha < \lambda_1$
 and suppose we have chosen
 $a_\beta \in S_\beta'$ for each $\beta < \alpha$.
For each $\beta < \alpha$, let
 $B_\beta := \{ b \in S_\alpha : a_\beta || b \}$.
Each set $B_\beta$ has measure zero,
 and there are $< \lambda_1$ of them.
By the $\lambda_1$-completeness of the ultrafilter,
 there must be an element of $S_\alpha'$ not in any
 $B_\beta$ for $\beta < \alpha$.
Let $a_\alpha$ be any such element.
The sequence
 $\langle a_\alpha : \alpha < \lambda_1 \rangle$
 works as desired.
\end{proof}

The next lemma gives a flavor
 of how we can shrink to either
 get a red or a blue node.

\begin{lemma}[Red-Blue Concentration]
\label{redblue}
Let $\lambda_1 < \lambda_2$ be such that
 $\lambda_1$ is a measurable cardinal and
 there is a uniform $\lambda_1$-complete ultrafilter on $\lambda_2$.
Let $T \in \mbb{P}$ and $t \in T$ be the stem of $T$.
Assume
 $|\mathrm{Succ}_T(t)| = \lambda_1$.
Let $\mc{U}$ be an ultrafilter on
 $\mathrm{Succ}_T(t)$ that comes from
 a fixed normal ultrafilter on $\lambda_1$
 and a fixed bijection
 between $\lambda_1$ and $\mathrm{Succ}_T(t)$.
So $\mc{U}$ is $\lambda_1$-complete and
 concentrates on $\mathrm{Succ}_T(t)$.
For each $c \in \mathrm{Succ}_T(t)$,
 let $s_c \sqsupseteq c$ be the shortest
 splitting extension of $c$, and assume that in fact
 $|\mathrm{Succ}_T(s_c)| = \lambda_2$ and there is
 a uniform $\lambda_1$-complete ultrafilter $\mc{U}_c$
 which concentrates on $\mathrm{Succ}_T(s_c)$.
Assume further that for each
 $c \in \mathrm{Succ}_T(t)$,
 $s_c$ is either a red node of $T$
 or a blue node of $T$.
Then there is a set
 $C \subseteq \mathrm{Succ}_T(t)$ in $\mc{U}$
 and for each $c \in C$ a tree
 $T_c \subseteq T | c$
 such that when we define
 $T' := \bigcup_{c \in C} T_c$,
 then exactly one of the following holds:
\begin{itemize}
\item[1)] The values of
 $\psi_{T'}(c)$ for $c \in C$
 are pairwise incomparable, so
 $t$ is a blue node of $T'$;
\item[2)] The values of
 $\psi_{T'}(c)$ for $c \in C$
 are all the same,
 so $t$ is a red node of $T'$.
Also, for each $c \in C$,
 we have that
 $\mc{U}_c$ concentrates on
 $\textrm{Succ}_{T'}(s_c)$
 and $s_c$ is a red node of $T'$.
This implies that $\psi_{T'}(\tilde{c})$
 is the same for each
 $\tilde{c} \in \textrm{Succ}_{T'}(s_c)$ and
 $c \in \textrm{Succ}_{T'}(t)$.
\end{itemize}
\end{lemma}
\begin{proof}
First use the fact that $\mc{U}$
 is an ultrafilter on $\textrm{Succ}_T(t)$
 to get a set $C_0 \subseteq \textrm{Succ}_T(t)$
 in $\mc{U}$ such that the nodes
 $s_c$ for $c \in C_0$ are either all blue in $T$
 or all red in $T$.

Suppose the nodes $s_c$
 (for $c \in C$) are all blue in $T$.
Set $C := C_0$.
Then use the lemma above
 (the Blue Selection Lemma) to pick one child
 $\tilde{c}_c$ of each $s_c$ (for $c \in C$)
 such that the resulting sequences
 $\psi_{T}(\tilde{c}_c)$ are all
 pairwise incomparable.
It is here that we use the fact that the
 ultrafilters $\mc{U}_c$ are $\lambda_1$-complete.
Now define each $T_c \subseteq T | c$ to be
 $T_c := T | {\tilde{c}_c}$.
Define $T'$ to be
 $\bigcup_{c \in C} T_c$.
We have $
 \psi_{T}(\tilde{c}_c) =
 \psi_{(T | \tilde{c}_c)}(\tilde{c}_c) =
 \psi_{T_c}(c) =
 \psi_{T'}(c)$.
Since the $\psi_T(\tilde{c}_c)$
 for $c \in C$ are pairwise incomparable,
 then the $\psi_{T'}(c)$ for $c \in C$ are pairwise
 incomparable, so
 1) holds.

Suppose now that the nodes $s_c$
 (for $c \in C_0$) are all red in $T$.
Given $c \in C_0$,
 $\psi_{T}(\tilde{c})$ does not depend on
 which $\tilde{c} \in \textrm{Succ}_T(s_c)$ is used,
 so each $\psi_T(\tilde{c})$
 for $\tilde{c} \in \mathrm{Succ}_T(s_c)$
 in fact equals
 $\psi_T(s_c)$.
We also have $\psi_T(s_c) = \psi_T(c)$
 for each $c \in C_0$.
We will now use the assumption that
 $\lambda_1$ is a measurable cardinal.
Since $\lambda_1$ is a measurable cardinal,
 if $\mc{V}$ is any normal ultrafilter on $\lambda_1$,
 then $\lambda_1 \rightarrow (\mc{V})^2_2$.
Thus, there is a set
 $C_1 \subseteq C_0$ in $\mc{U}$ such that
 the sequences $\psi_T(c)$ for $c \in C_1$
 are either all pairwise comparable or all pairwise
 incomparable.

\underline{Case 1}:
If they are all pairwise comparable,
 then because they might have different lengths,
 use the $\omega_1$-completeness of $\mc{U}$
 to get a set $C_2 \subseteq C_1$ in $\mc{U}$
 such that the $\psi_T(c)$ for $c \in C_2$
 are identical
 (by getting them to have the same lengths).
Set $C := C_2$ and set each $T_c \subseteq T | c$
 to be $T_c := T | c$ (no thinning of the
 subtrees is necessary).
We have that 2) holds.

\underline{Case 2}:
If they are pairwise incomparable,
 then set $C := C_1$ and
 set each $T_c \subseteq T | c$
 to be $T_c := T | c$
 (no thinning of subtrees is necessary).
We have that 1) holds.
\end{proof}

We are now ready for the fundamental
 lemma needed to analyze the minimality of $\mbb{P}$
 (for functions with domain $\omega$).
\begin{lemma}[Blue Production for $\dot{A} : \omega \to \check{V}$]
\label{blueproduction}
Assume $\cf(\kappa) = \omega$.
Fix $n < \omega$.
Suppose $\kappa_{n} < \kappa_{n+1} < ...$
 are all measurable cardinals.
Let $T \in \mbb{P}$ with stem $s \in T$.
Let $\dot{A}$ be such that
 $T \forces \dot{A} : \omega \to \check{V}$ and
 $T \forces \dot{A} \not\in \check{V}$.
Suppose $s$ has exactly
 $\kappa_n$ children in $T$.
Then there is some perfect $W \subseteq T$ such that
 $s$ has $\kappa_n$ children in $W$ and
 $s$ is blue in $W$.
\end{lemma}
\begin{proof}
To prove this result,
 we will frequently pick some node
 in a tree and fix an ultrafilter which
 concentrates on the set of its children in that tree.
When we shrink the tree further,
 we will ensure that as long as
 the node has $> 1$ child, then the ultrafilter
 will still concentrate on the set of its children.
To index this,
 we will have partial functions which map
 nodes to ultrafilters.
We will start with the empty partial function.
Once we attach an ultrafilter to a node,
 will never attach a different ultrafilter
 to the same node later.

We will define a (partial) function $\Phi$ recursively.
As input it will take in a tuple
 $\langle Q, t, \vec{\mc{U}}, m, k \rangle$,
 and as output it will return
 $\langle Q', \vec{\mc{U}}' \rangle$.
$Q \supseteq Q'$ are perfect trees.
$\vec{\mc{U}} \subseteq \vec{\mc{U}}'$ are partial functions,
 mapping nodes to ultrafilters.
$m$ and $k$ are both numbers $< \omega$.
$Q$ has stem $t$
 (passing the stem $t$ to the function $\Phi$ is redundant,
 but we do it for emphasis).
The node $t \in Q$ has at least $\kappa_m$ children in $Q$,
 it is in $Q'$, and
 it has exactly $\kappa_m$ children in $Q'$.
Moreover,
 $t \in \dom(\vec{\mc{U}}')$ and
 $\vec{\mc{U}}'(t)$ concentrates on
 $\mathrm{Succ}_{Q'}(t)$.
The number $k$ is how many
 remaining recursive steps to take.
Finally, one of the following holds
 (note the additional purpose of $m$ and $k$):
\begin{itemize}
\item[1)] $t$ is blue in $Q'$, or
\item[2)] $t$ is red in $Q'$ and
 $\dom(\psi_{Q'}(t)) \ge m+k$.
\end{itemize}
That is, if $t$ is red in $Q'$,
 then at least the first
 $m + k$ values of $\dot{A}$ are decided
 by $(Q' | t) = Q'$.
We will now define $\Phi$ recursively on $k$:

\underline{$\Phi(Q,t,\vec{\mc{U}},m,0)$}:
First, remove children of $t$
 so that in the resulting tree
 $Q_0 \subseteq Q$, $t$ has \textit{exactly}
 $\kappa_m$ children.
If this is impossible,
 then the function is being used incorrectly,
 so leave the function undefined
 on this input.
At this point, we should have
 $t \not\in \dom(\vec{\mc{U}})$, otherwise the function
 is being used incorrectly.
Let $\mc{U}$ be a $\kappa_m$-complete ultrafilter
 on $\mathrm{Succ}_{Q_0}(t)$
 that is induced by a normal ultrafilter
 on $\kappa_m$
 and a bijection between
 $\mathrm{Succ}_{Q_0}(t)$ and $\kappa_m$.
Attach this ultrafilter to $t$ by defining
 $\vec{\mc{U}}' := \vec{\mc{U}} \cup \{ (t, \mc{U}) \}$.

We now must define $Q' \subseteq Q$.
For each $c \in \mathrm{Succ}_{Q_0}(t)$,
 let $U_c \subseteq Q_0 | c$ be some condition
 which decides at least the first $m+0$ values of $\dot{A}$.
Let $Q_1 := \bigcup_c U_c$.
We have $Q_1 \subseteq Q_0$.
Of course, $\mathrm{Succ}_{Q_1}(t) = \mathrm{Succ}_{Q_0}(t)$.
Consider the coloring $b : [\mathrm{Succ}_{Q_1}(t)]^2 \to 2$
 defined by $b(c_1, c_2) = 1$ iff
 $\psi_{Q_1}(c_1)$ and $\psi_{Q_1}(c_2)$
  are comparable, and $b(c_1, c_2) = 0$ otherwise.
Since the ultrafilter $\mc{U}$
 which concentrates on $\mathrm{Succ}_{Q_1}(t)$
 is induced by a normal ultrafilter on $\kappa_m$,
 fix a set $C_0 \subseteq \mathrm{Succ}_{Q_1}(t)$ in $\mc{U}$
 that will homogenize the coloring $b$.
Hence the sequences
 $\psi_{Q_1}(c)$ for $c \in C_0$ are either pairwise
 incomparable or pairwise comparable.

Let $Q_2 \subseteq Q_1$ be the tree obtained by only
 removing the children of $t$ that are not in $C_0$.
If the sequences $\psi_{Q_2}(c) = \psi_{Q_1}(c)$
 for $c \in C_0$
 are pairwise incomparable,
 then we are done by defining $Q' := Q_2$
 ($t$ is blue in $Q_2$).
If not, then apply the pigeon hole principle for
 $\omega_1$-complete ultrafilters to get a set
 $C_1 \subseteq C_0$ in $\mc{U}$ such that all
 $\psi_{Q_2}(c)$ sequences for $c \in C_1$
 are the \textit{same}.
Let $Q_3 \subseteq Q_2$ be the tree obtained from $Q_2$ by only
 removing the children of $t$ that are not in $C_1$.
We are done by defining $Q' := Q_3$
 ($t$ is red in $Q_3$ and $Q_3$ decides
 at least the first $m+0$ values of $\dot{A}$).

\underline{$\Phi(Q,t,\vec{\mc{U}},m,k+1)$}:
It must be that $t$ has $\kappa_m$ children in $Q$,
 otherwise the function is being used incorrectly.
Also, it must be that $t \in \dom(\vec{\mc{U}})$
 and $\vec{\mc{U}}(t)$ concentrates on
 $\mathrm{Succ}_{Q}(t)$.

Temporarily fix a $c \in \mathrm{Succ}_{Q}(t)$.
Let $s_c \sqsupseteq c$ be a minimal
 extension in $Q$ with $\ge \kappa_{m+1}$ children
 (if $k > 0$, by the way the function is used,
 the node $s_c$ will be unique).
Let $U_c := Q | s_c$.
Let $\langle U_c', \vec{\mc{U}}_c \rangle :=
 \Phi(U_c, s_c, \vec{\mc{U}}, m+1, k)$.
We have that $s_c \in \dom(\vec{\mc{U}}_c)$
 and $\vec{\mc{U}}_c(s_c)$ is a
 $\kappa_{m+1}$-complete ultrafilter that concentrates
 on the size $\kappa_{m+1}$ set of children of
 $s_c$ in $U_c'$.
Also, $s_c$ is either a blue node of $U_c'$,
 or it is a red node of $U_c'$ and $U_c'$
 decides at least the first $(m+1)+k$ elements of $\dot{A}$.
Now unfix $c$.
Define $\vec{\mc{U}}' := \bigcup_c \vec{\mc{U}}_c$.
Let $Q_0 := \bigcup_c U_c' \subseteq Q$.

Use the fact that $\vec{\mc{U}}(t)$
 is an ultrafilter that concentrates on
 $\mathrm{Succ}_{Q_0}(t)$
 to get a set
 $C_0 \subseteq \mathrm{Succ}_{Q_0}(t)$
 in $\vec{\mc{U}}(t)$ such that
 the nodes $s_c$ for $c \in C_0$ are either
 all red in $Q_0$ or all blue in $Q_0$.
We will break into cases.

\underline{Case 1}:
First, consider the case that
 the nodes $s_c$ for $c \in C_0$ are all blue
 in $Q_0$.
Use Lemma~\ref{blueselection} (Blue Selection)
 to get, for each $c \in C_0$, a node
 $\tilde{c}_c \in \mathrm{Succ}_{Q_0}(s_c)$ such that
 the sequences $\psi_{Q_0}(\tilde{c}_c)$ are pairwise
 incomparable.
Note that for each $c \in C_0$,
 $\psi_{Q_0 | \tilde{c}_c}(c) =
  \psi_{Q_0}(\tilde{c}_c)$.
Let $Q_1 := \bigcup_{c \in C_0} (Q_0 | \tilde{c}_c)
 \subseteq Q_0$.
We have that $t$ is a blue node of $Q_1$.
Defining $Q' := Q_1$, we are done.

\underline{Case 2}:
The other case is that the nodes $s_c$
 for $c \in C_0$ are all red.
Let $b : [C_0]^2 \to 2$
 be the coloring defined by
 $b(c_1, c_2) = 1$ iff
 $\psi_{Q_0}(c_1)$ and $\psi_{Q_0}(c_2)$
 are comparable,
 and $b(c_1, c_2) = 1$ otherwise.
Since the ultrafilter $\vec{\mc{U}}(t)$
 which concentrates on $C_0$
 is induced by a normal ultrafilter on $\kappa_m$,
 fix a set $C_1 \subseteq C_0$ in $\vec{\mc{U}}(t)$
 that will homogenize the coloring $b$.
Hence the sequences $\psi_{Q_0}(c)$ for $c \in C_1$
 are either all comparable
 or all incomparable.

If they are pairwise incomparable,
 then define
 $Q' := \bigcup_{c \in C_1} (Q_0 | c)
 \subseteq Q_0$.
The node $t$ is blue in $Q'$, and we are done.
If they are pairwise comparable,
 then use the fact that $\vec{\mc{U}}(t)$
 is $\omega_1$-complete
 to get a set $C_2 \subseteq C_1$ in
 $\vec{\mc{U}}(t)$ such that the sequences
 $\psi_{Q_0}(c)$ for $c \in C_2$ are all the same
 (by getting the
 sequences $\psi_{Q_0}(c)$ to have the same length,
 we get them to be identical).
Define $Q' :=
 \bigcup_{c \in C_2} (Q_0 | c) \subseteq Q_0$.
We have that $t$ is red in $Q'$.
From our definition of a red node,
 since each $s_c$ is a red node of $Q'$,
 it follows that for each $c \in C_2$
 and each $c' \in \mathrm{Succ}_{Q'}(s_c)$,
 we have $\psi_{Q'}(c) = \psi_{Q'}(c')$.
We said earlier that $U_c'$ decides at least
 the first $m + (k+1)$ elements of $\dot{A}$.
Thus, $Q'$ itself decides at least
 the first $m + (k+1)$ values of $\dot{A}$.
This completes the definition of $\Phi$.

With $\Phi$ defined,
 we will prove the lemma.
Let $\langle T_0, \vec{\mc{U}}_0 \rangle :=
 \Phi(T, s, \emptyset, n, 0)$.
If $s$ is blue in $T_0$, we are done
 by setting $W := T_0$.
If not, then
 $(T_0 | s) = T_0$ decides at least the first
 $n$ values of $\dot{A}$.
Next, let $\langle T_1, \vec{\mc{U}}_1 \rangle :=
 \Phi(T_0, s, \vec{\mc{U}}_0, n, 1)$.
If $s$ is blue in $T_1$, we are done
 by setting $W := T_1$.
If not, then
 $(T_1 | s) = T_1$ decides at least the first
 $n+1$ values of $\dot{A}$.
Next, let $\langle T_2, \vec{\mc{U}}_2 \rangle :=
 \Phi(T_1, s, \vec{\mc{U}}_1, n, 2)$, etc.

We claim that this procedure eventually terminates.
If not, then we have produced the sequences
 $T_0 \supseteq T_1 \supseteq T_2 \supseteq ...$
 (which is probably \textit{not} a fusion sequence) and
 $\vec{\mc{U}}_0 \subseteq \vec{\mc{U}}_1
 \subseteq \vec{\mc{U}}_2 \subseteq ...$.
Let $T_\omega := \bigcap_{i < \omega} T_i$.
If we can show that $T_\omega$ contains a perfect tree $\tilde{T}$,
 then we will have that $\tilde{T}$ decides at
 least the first $k$ values of $\dot{A}$
 for every $k < \omega$,
 which implies $\tilde{T} \forces \dot{A} \in \check{V}$,
 which is a contradiction.

First note that the stem $s$ of $T$ satisfies
 $s \in \dom(\vec{\mc{U}}_0)$
 and $s$ has $\vec{\mc{U}}_0(s)$ many children
 in each tree $T_i$.
Using the $\omega_1$-completeness
 of $\vec{\mc{U}}_0(s)$,
 $s$ has $\vec{\mc{U}}_0(s)$ many children
 in $T_\omega$, so in particular it has $\kappa_n$ children
 in $T_\omega$.

Now temporarily fix $c \in \mathrm{Succ}_{T_\omega}(s)$.
Let $s_c$ be the minimal extension of $c$ in $T_1$
 that has $\ge \kappa_{n+1}$ children.
Now $s_c$ will never become a blue node in $T_i$
 for any $i \ge 1$,
 because otherwise
 because $s$ remains red we would have that
 $s_c$ would get removed at some point
 and hence $c$ would get removed,
 contradicting that $c \in \mathrm{Succ}_{T_\omega}(s)$.
We can see by the ways trees are shrunk
 in the definition of $\Phi$ that the following holds:
 $\vec{\mc{U}_1}(s_c)$ is defined
 and for each $i \ge 1$,
 $s_c$ is in each $T_i$
 and
 $s_c$ has $\vec{\mc{U}_1}(s_c)$ many children in $T_i$.
So by the $\omega_1$-completeness of $\vec{\mc{U}_1}(s_c)$,
 $s_c$ has $\kappa_{n+1}$ children in $T_\omega$.

Continuing like this,
 here is the general pattern.
We let $S_0 = \{ s \}$.
Then, for $i \in \omega$
 having defined $S_i$,
 we define $S_{i+1}$ as follows:
 a node $c$ is in $S_{i+1}$ iff
 it is the minimal extension
 of a node in
 $$\bigcup \{ \mbox{Succ}_{T_\omega}(t) : t \in S_i \}$$
 that has $\kappa_{n + i + 1}$ children in $T_{i+1}$.
Let $\tilde{T}$ be the set of all initial segments
 of nodes in $\bigcup_i S_i$.
One can check that $\tilde{T}$ is a perfect subtree of $T_\omega$.
In fact, $\tilde{T} = T_\omega$.

%
\end{proof}

\begin{theorem}\label{thm.6.6}
Assume $\cf(\kappa) = \omega$.
Suppose the cardinals $\kappa_0 < \kappa_1 < ...$
 are all measurable.
Fix a condition $T \in \mbb{P}$.
Let $\dot{A}$ be a name such that
 $T \forces (\dot{A} : \omega \to \check{V})$ and
 $T \forces (\dot{A} \not\in \check{V})$.
Let $\dot{G}$ be a name for the generic object.
Then $T \forces \dot{G} \in \check{V}(\dot{A})$.
\end{theorem}

\begin{proof}
It suffices to find a condition $T' \le T$
 satisfying the hypotheses of
 Lemma~\ref{bluecoding} (Blue Coding).
We will construct $T'$ by performing fusion.

Let $T_\emptyset \le T$ be such that the stem
 $t_\emptyset \in T_\emptyset$ is $0$-splitting.
Apply Lemma~\ref{blueproduction}
 (Blue Production)
 to the tree $T_\emptyset$ and the node
 $t_\emptyset \in T_\emptyset$
 to get $T_\emptyset' \le T_\emptyset$.
Now $t_\emptyset$ is blue
 and $0$-splitting in $T_\emptyset'$.
Hence, the unique $0$-splitting node of $T_\emptyset'$ is blue.
Define $T_0 := T_\emptyset'$,
 the first element of our fusion sequence.

Now, fix any $c \in \mathrm{Succ}_{T_0}(t_\emptyset)$.
Let $T_c \le (T_0 | c)$ be
 such that there is a (unique) $1$-splitting node
 $t_c \sqsupseteq c$ in $T_c$.
Apply Lemma~\ref{blueproduction}
 (Blue Production)
 to the tree $T_c$ and the node $t_c$
 to get $T_c' \le T_c$.
Now $t_c$ is blue and $1$-splitting in $T_c'$.
Unfixing $c$, let us define
 $T_1 := \bigcup \{ T_c' : c \in \mathrm{Succ}(T_c',t_\emptyset) \}$.
We have $T_1 \le T_0$,
 every child of $t_\emptyset$ is in $T_1$
 (so in particular it is $0$-splitting),
 and every $1$-splitting node of $T_1$ is blue.

We may continue like this to get the fusion sequence
 $T_0 \supseteq T_1 \supseteq T_2 \supseteq ...$.
Define $T'$ to be the intersection of this sequence.
We have that $T'$ is in weak splitting normal form
 (every node with $>1$ child is $n$-splitting for some $n$).
Since being blue is preserved when we pass to a stronger condition,
 every splitting node of $T'$ is blue.
We may now apply
 Lemma~\ref{bluecoding} (Blue Coding),
 and the theorem is finished.
\end{proof}

So $\mathbb{P}$ is minimal with respect to
 $\omega$-sequences of ordinals,
 but by what we found earlier it is not minimal:

\begin{cor}\label{cor.6.7}
The forcing $\mbb{P}$ does not add a minimal degree of constructibility. 
\end{cor}

\begin{proof}
Let $\mbb{B}$ be the regular open completion of $\mbb{P}$.
In the previous section,
 we showed that there is a complete
 embedding of $\mc{P}(\omega)/\Fin$
 into $\mbb{B}$.
Let $G$ be generic for $\mbb{P}$ over $V$.
Let $H \in V[G]$ be generic for
 $\mc{P}(\omega)/\Fin$ over $V$.
Since $\mc{P}(\omega)/\Fin$ is
 countably complete, it does not add any
 new $\omega$-sequences, so $G \not\in V[H]$.
On the other hand, we have $H \not\in V$.
Thus,
 $V \subsetneq V[H] \subsetneq V[G]$,
 so the forcing is not minimal.
\end{proof}


\section{Uncountable height counterexample and open problems}
\label{secuncheight}

To conclude the paper, we present an example of what can go wrong
 when one tries to generalize some of the results in the previous
 sections to singular cardinals $\kappa$  with uncountable cofinality.

Assuming $\cf(\kappa) > \omega$,
 we will first construct a pre-perfect tree
 $T \subseteq N$
 such that $[T]$ has size $\kappa$.

\begin{lemma}
Let $g : \mathrm{Ord} \to 2$ be a function.
Given an ordinal $\gamma$, let
 $$S_{g \restriction \gamma}
 := \{ \alpha < \gamma : g(\alpha) = 1 \}.$$
Let $\Phi_{< \gamma}$ be the statement that
 for each limit ordinal $\alpha < \gamma$,
 $g$ equals $0$ for a final segment of $\alpha$.
Let $\Phi_{\gamma}$ be the analogous statement
 but for all $\alpha \le \gamma$.
The following hold: 
\begin{itemize}
\item[1)] If $\Phi_\gamma$, then $S_{g \restriction \gamma}$ is finite.
\item[2)] If $\Phi_{<\gamma}$ and $\cf(\gamma) \not= \omega$,
 then $S_{g \restriction \gamma}$ is finite.
\item[3)] If $\Phi_{<\gamma}$, then $S_{g \restriction \gamma}$
 is countable.
\end{itemize}
\end{lemma}
\begin{proof}
We can prove these by induction on $\gamma$.
If $\gamma = 0$, there is nothing to do.
Now assume that $\gamma$ is a successor ordinal.
If we assume $\Phi_{<\gamma}$, then
 $\Phi_{\gamma-1}$ is true so
 by the inductive hypothesis and the fact that
 $$|S_{g \restriction \gamma}| \le
 |S_{g \restriction (\gamma-1)}| + 1,$$
 $S_{g \restriction \gamma}$ is finite.

Now assume that $\cf(\gamma) = \omega$.
Let $\langle \gamma_n : n \in \omega \rangle$
 be a sequence cofinal in $\gamma$.
Note that
 $$S_{g \restriction \gamma} =
 \bigcup_{n \in \omega}
 S_{g \restriction \gamma_n} =
 S_{g \restriction \gamma_0} \cup
 \bigcup_{n \in \omega} (
 S_{g \restriction \gamma_{n+1}} -
 S_{g \restriction \gamma_n} ).$$
Thus, if we assume $\Phi_{<\gamma}$,
 then $\Phi_{\gamma_n}$ holds for each $n$,
 so by the induction hypothesis
 each $S_{g \restriction \gamma_n}$ is finite,
 so $S_{g \restriction \gamma}$ is countable.
If additionally we assume $\Phi_{\gamma}$,
 then it must be that all but finitely
 many of the
 $S_{g \restriction \gamma_{n+1}} -
 S_{g \restriction \gamma_n}$
 are empty, so $S_{g \restriction \gamma}$ is finite.

Finally, assume $\cf(\gamma) > \omega$
 and $\Phi_{< \gamma}$.
For each limit ordinal $\alpha < \gamma$,
 let $f(\alpha) < \alpha$ be such that $g$ is $0$
 from $f(\alpha)$ to $\alpha$.
By Fodor's Lemma, fix some $\beta < \gamma$
 such that $f^{-1}(\{ \beta \}) \subseteq \gamma$
 is a stationary
 subset of $\gamma$.
Since $f^{-1}(\{ \beta \})$
 is cofinal in $\gamma$,
 we see that $g$ is $0$ from
 $\mu := \min f^{-1}(\{ \beta \})$ to $\gamma$.
Thus,
 $S_{g \restriction \gamma} =
  S_{g \restriction \mu}$.
The set $S_{g \restriction \mu}$
 is finite because of $\Phi_{\mu}$
 and the induction hypothesis,
 so we are done.
\end{proof}

We can now get the desired
 counterexample:
\begin{counterexample}\label{counterex}
\label{uncountablecounterex}
Assume $\cf(\kappa) > \omega$.
There is a pre-perfect tree $T \subseteq N$
 such that $[T]$ has size $\kappa$,
 and hence $[T]$ is not perfect.
\end{counterexample}

\begin{proof}
We will define $T \subseteq N$.
Define the $\alpha$-th level of $T$ as follows:
\begin{itemize}
\item[1)] if $\alpha = 0$,
 then the level consists of only
 the root $\emptyset$.
\item[2)] If $\alpha = \beta + 1$,
 then a node is in the $\alpha$-th
 level of $T$ iff it is the successor
 in $N$ of a node in the $\beta$-th
 level of $T$.
\item[3)] If $\alpha$ is a limit ordinal,
 then a node $t$ is in the $\alpha$-th level
 of $T$ iff every proper initial segment of $t$
 is in $T$ and $t(\beta) = 0$
 for a final segment of $\beta$'s
 less than $\alpha$.
\end{itemize}
First, let us verify that $T$ is non-stopping.
Consider any node $t \in T$.
Let $f \in X$ be the unique function
 that extends $t$ such that
 $f(\alpha) = 0$ for all $\alpha$ in
 $\dom(f) - \dom(t)$.
We see that $f$ is a path through $T$.

We will now show that $[T]$ has size
 $\le \kappa$.
Consider any $f \in [T]$.
Let $g : \cf(\kappa) \to 2$ be the function
 $$g(\alpha) :=
 \begin{cases}
 0 & \mathrm{if } f(\alpha) = 0, \\
 1 & \mathrm{otherwise.}
 \end{cases}$$
By the definition of $T$ and the lemma
 above, it must be that
 $\{ \alpha < \cf(\kappa) : g(\alpha) = 1 \}$
 is finite.
Recall that for each $\alpha < \cf(\kappa)$,
 there are at most $\kappa_\alpha$
 possible values for $f(\alpha)$.
Now, a simple computation shows that
 there are at most $\kappa$ such paths $f$
 associated to a given $g$
 (in fact, there are exactly $\kappa$).
\end{proof}

This counterexample points to the need for some further requirements on the trees when $\kappa$ has uncountable cofinality.
Such obstacles will likely be overcome by assuming that splitting levels on branches are club, as in \cite{Kanamori80} and  \cite{Brown/Groszek06}, as this will provide fusion for $\cf(\kappa)$ sequences of trees.
We ask, which distributive laws hold and which ones fail for the Boolean completions of the families of perfect tree forcings similar to those in this paper for singular $\kappa$ of uncountable cofinality, but requiring club splitting, or some other splitting requirement which ensures $\cf(\kappa)$-fusion.
More generally,

\begin{question}
Given a regular cardinal $\lambda$,
for which cardinals $\mu$ is there a complete Boolean algebra in which 
for all $\nu<\mu$, the $(\lambda,\nu)$-d.l.\ holds but the $(\lambda, \mu)$-d.l.\ fails?
\end{question}

Similar questions remain open for three-parameter distributivity.

\bibliographystyle{amsplain}
\bibliography{references}

\end{document}